\documentclass[10pt]{amsart}

\usepackage{tikz}
\usetikzlibrary{decorations.pathreplacing}
\usepackage{amsfonts}
\usepackage{xcolor}
\usetikzlibrary{positioning}

\usetikzlibrary{arrows.meta}

\usepackage{amssymb}

\usepackage{multirow}

\usepackage{graphicx}

\usepackage{color}

\usepackage{amsmath}

\usepackage{graphicx}

\usepackage{color}

\usepackage{stackengine}
\usepackage{xcolor}
\usepackage{scalerel}
\usepackage{polynom}
\usepackage{MnSymbol}
\usepackage{textgreek}
\usepackage{xfrac}

\newtheorem{theorem}{Theorem}[section]
\newtheorem{lemma}[theorem]{Lemma}

\theoremstyle{definition}
\newtheorem{definition}[theorem]{Definition}

\newtheorem{proposition}[theorem]{Proposition}
\newtheorem{corollary}[theorem]{Corollary}

\theoremstyle{remark}

\title{Leaky Forcing: A New Variation of Zero Forcing}
\author{Shannon Dillman and Franklin Kenter}\thanks{Department of Mathematics, United States Naval Academy, 
Annapolis, MD 21402. Corresponding author: kenter@usna.edu}

\usepackage{amsmath}
\usepackage{algorithm}
\usepackage[noend]{algpseudocode}

\makeatletter
\def\BState{\State\hskip-\ALG@thistlm}
\makeatother

\usepackage{bbm}

\begin{document}
\maketitle
\begin{abstract}
Zero forcing is a one-player game played on a graph. The player chooses some set of vertices to color, then iteratively applies a color change rule: If all but one of a colored vertex's neighbors are colored, color (i.e.â ``force'' ) the remaining uncolored
neighbor. Generally, the goal is to find the minimum number of vertices to initially color such that all vertices eventually become colored. Recently, equivalent formations of zero forcing have been developed in different settings including sensor allocation to solve linear systems (K.-Lin 2018), controllability in follower-leader dynamics (Monshizadeh-Zhang-Camlibel 2014), and edge covering in specific hypergraphs (Brimkov-Fast-Hicks 2016).

While many variations of zero forcing are motivated by an associated minimum rank problem, these new formulations give new inspiration for new meaningful zero forcing variants.
In our case, we study a new variation based on the linear algebraic interpretation mentioned above. In particular, what if there is a juncture in a network that has a leak, and, hence, is unreliable to facilitate solving a linear system on the network?
In the context of zero forcing this corresponds to the following variation we call $\ell$-forcing: Given $\ell$, find a set of vertices such that for {\it any} set of $\ell$ vertices
that are unable to force, all vertices will still be colored.
We compute the $\ell$-forcing number for selected families of graphs including grid graphs. Perhaps surprisingly, we find examples where additional edges make the graph more ``resilient'' to these leaks. Further, we also implement known computational methods for our new leaky forcing variation. 
\end{abstract}





\section{Introduction}

Zero forcing is a game played on a graph, $G$. The game was originally developed to bound the maximum nullity rank (or equivalently, the minimum rank) over the set of matrices with a corresponding zero-nonzero pattern, $\mathcal{S}(G)$ \cite{zf_aim}.
In the zero forcing game, some vertices start as colored and others uncolored, and colored vertices ``force'' uncolored vertices to become colored under the color-change rule: If all but one of a colored vertex's neighbors are colored, the remaining uncolored neighbor becomes colored. A set of vertices is a {\it zero forcing set} if, when initially colored, all vertices in $G$ will eventually become colored after iteratively applying the color-change rule. The general goal is to find the minimum zero forcing set.

Perhaps surprisingly, zero forcing and its variations have a habit of being rediscovered in different fields and applications. These include power systems and sensor allocation \cite{pmu,zf.powerdom1, zf.powerdom2}, control of quantum systems \cite{zf_qc}, and controllability of leader-follower systems \cite{zf_control}.

Despite all of these applications, most variations of zero forcing arise from the original context of the minimum rank problem: If one places additional restrictions or relaxations on the matrix pattern, those restrictions translate to variations in the color-change rule. For instance, restricting the desired matrices to positive semidefinite and/or skew-symmetric matrices adds additional color-change rules making it easier for vertices to force \cite{zf_skew, zf_psd}. On the other hand, the applications mentioned above indicate that there are other potential meaningful variations of zero forcing beyond the minimum rank problem, as we will study here.

The variation we study in this article is based on the following equivalent notion of zero forcing:

\begin{theorem}[K.-Lin, \cite{zf_sample}]
For a graph $G$, a set of of vertices, $S$, is a zero forcing set if and only if for any $\mathbf{A} \in \mathcal{S}(G)$, the partial vector $\mathbf{x}_S$ (i.e., $\mathbf{x}$ restricted to entries corresponding to $S$) is always sufficient to solve $\mathbf{Ax} = \mathbf{b}$ uniquely (if a solution exists).
\end{theorem}

Let us put this into context. Suppose you know the pattern of interactions of a system and you wish to place sensors on the nodes in order to observe the entire system once it is realized. Using natural laws (such as Kirchhoff's Laws), once the system is realized, it is possible to deduce the status of some nodes that do not have a sensor; hence, it is not necessary to place sensors at every node. However, when placing the sensors, you do not know the exact details other than the rough pattern of interactions. The theorem above says the {\it only} way to guarantee full accurate observability is to place sensors on nodes corresponding to a zero forcing set. Indeed, this is exactly the concept behind the application of zero forcing in power grids and quantum systems: One cannot wait until the system is realized to determine what aspects to monitor or measure. Rather, one must be ready to go beforehand, only knowing the rough pattern of interactions.

The variation we introduce in this study, ``leaky forcing,'' is based upon a variation in the scenario above. In a literal sense, we ask the question: What if there is a leak in the network? (We further define leaky forcing in Section \ref{leak intro}, Definition \ref{leakdef}). If a leak exists in a network at a specific node, then the natural laws could not correctly be applied to further deduce the status of other nodes. Since each force corresponds to applying a linear equation (or physical law) to deduce further information  \cite{zf_aim, zf_sample}, and a leak corresponds to a vertex that is unable to force, we must find a way to force despite these leaks. The goal is to find a set of initially colored vertices that will be able to force regardless of the location of some predetermined number of leaks.

In this article, we contribute the following:
\begin{itemize}
\item We provide preliminaries including a definition of leaky forcing and the {\it $\ell$-forcing number} of a graph, $Z_{(\ell)}(G)$. (Section \ref{leak intro}, Definition \ref{leakyforcingnumberdef})
\item We derive the exact value of $Z_{(\ell)}(G)$ for select families of graphs, including paths, cycles, and wheels. (Section \ref{sec.families})
\item We determine exact values and bounds for Cartesian (Box) products of graphs $Z_{(\ell)}(G \Box H)$ including the specific families such as grid graphs and hypercubes. (Section \ref{sec.boxproducts})
\item We construct and implement an integer program algorithm to compute $Z_{(\ell)}(G)$ exactly. (Section \ref{sec.alg})
\end{itemize}

\section{Preliminaries, Introduction to Leaks} \label{leak intro}


\begin{definition}[Graph Leaks] \label{leakdef} ~\\
A {\it leak} in a graph $G$ is defined to be one additional vertex adjacent to only one vertex of the original graph $G$ by one new edge.
\end{definition}

One of our goals is to define this game played on graphs with ``leaks." Leaks, which become barriers to the coloring of the graph since they are additional uncolored neighbors and therefore hinder the color-change rule, are assumed to appear in the graph after the set of initially colored vertices is determined and before the color-change rule has been applied to the original graph (i.e., leaks cannot be colored). For a graph on $n$ vertices, there could be between $0$ and $n$ leaks, inclusive, as any vertex in the graph could have a leak and multiple leaks on the same vertex are treated as a single leak.

The vertices of the original graph which have leaks then find themselves with a new uncolored neighbor, just as a pipe with a leak unexpectedly finds itself giving water a new direction to travel. Our task then, given a predetermined amount of possible leaks, $\ell$, to a graph $G$, is to find a set of initially colored vertices so that regardless of where the $\ell$ leaks appear in $G$, all of $G$ still becomes colored under the color-change rule.

\begin{definition}[$\ell$-Forcing Number of a Graph $G$, $Z_{(\ell)}(G)$] \label{leakyforcingnumberdef} ~\\
For a graph $G$ and nonnegative integer $\ell$, the $\ell$-forcing number, $Z_{(\ell)}(G)$, is the minimum number of vertices which must be initially colored to ensure $G$ can be completely colored despite the addition of {\it any} $\ell$ leaks.
\end{definition}

It is worth noting that $Z_{(0)}(G)$ is the zero forcing of $G$ in the original sense.

We will look at this $\ell$-forcing number on various families of graphs in Section \ref{sec.families} and in Section \ref{sec.boxproducts}. Let us quickly define these graphs and notation. We will use $P_n$ to denote the path graph on $n$ vertices and $C_n$ to denote a cycle graph on $n$ vertices. $K_n$ is the complete graph on $n$ vertices.
The wheel graph, $W_n$, is a graph on $n+1$ vertices constructed from $C_n$ and adding an additional vertex adjacent to all the vertices on the cycle. A full definition of the Cartesian (Box) product of graphs, $G \Box H$ is provided in Section \ref{sec.boxproducts}.

The set of vertices which must be initially colored to ensure the graph is fully colored after application of the color-change rule is here called the {\it $\ell$-forcing set.} The {\it $\ell$-forcing number} is the size of the minimum $\ell$-forcing set.

We now provide some crucial basic results for our study of leaky forcing.

\begin{proposition} \label{zf_chain} For any graph $G$,
\[Z_{(0)}(G) \leq Z_{(1)}(G) \leq Z_{(2)}(G) \leq \dots \leq Z_{(n)}(G).\]
\end{proposition}

\begin{proof}
Since leaks disrupt the coloring process, having a set that can color with $n$ leaks is sufficient to color with $\ell \leq n$ leaks. If removing a leak created any change to the coloring process, it would be that without that leak, a vertex that before was unable to force is now able to force.
\end{proof}


\begin{lemma}
\label{ldegree}
For a graph $G$, any $\ell$-forcing set will contain at least those vertices in $G$ with degree $\ell$ or less.
\end{lemma}

\begin{proof}
Suppose that $G$ has $\ell$ leaks but not every vertex of degree $\ell$ or less is initially colored. Let $v$ be one such vertex with $\ell$ or fewer neighbors. Then, after the initial coloring of $G$ suppose every neighbor of $v$ has a leak. Then, none of those neighbors will be able to force $v$ to be colored and therefore $G$ could not be colored with that $\ell$-forcing set. Therefore, the $\ell$-forcing set for a graph with $\ell$ leaks must contain at least the vertices in $G$ with degree $\ell$ or less.
\end{proof}

\begin{center}
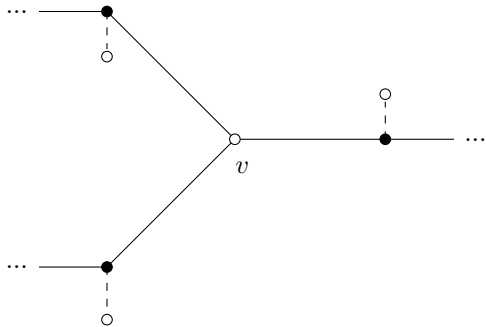
\begin{figure}[!ht]
\begin{tikzpicture}

\draw[thin] (0,0)--(-1.7,1.7);
\draw[thin] (0,0)--(-1.7,-1.7);
\draw[thin] (0,0)--(2,0);
\draw[dashed] (2,0)--(2,0.5);
\draw[dashed] (-1.7,1.7)--(-1.7,1.2);
\draw[dashed] (-1.7,-1.7)--(-1.7,-2.3);

\draw[thin] (2,0)--(3,0);
\draw[thin] (-1.7,1.7)--(-2.7,1.7);
\draw[thin] (-1.7,-1.7)--(-2.7,-1.7);

\draw (0,0) ++(285:0.38cm) node [fill=white,inner sep=4pt](a){$v$};

\draw[fill=white] (0,0) circle (2pt);
\draw[fill] (-1.7,-1.7) circle (2pt);
\draw[fill] (-1.7,1.7) circle (2pt);
\draw[fill] (2,0) circle (2pt);

\draw[fill=white] (-1.7,-2.4) circle (2pt);
\draw[fill=white] (2,0.6) circle (2pt);
\draw[fill=white] (-1.7,1.1) circle (2pt);

\draw (-2.7,1.7) ++(180:0.2cm) node [fill=white,inner sep=4pt](a){...};
\draw (-2.7,-1.7) ++(180:0.2cm) node [fill=white,inner sep=4pt](a){...};
\draw (3,0) ++(000:0.2cm) node [fill=white,inner sep=4pt](a){...};

\end{tikzpicture}
\caption{An example of Lemma \ref{ldegree}. There is no way for $v$ to become colored.}
\label{degreeexample}
\end{figure}
\end{center}

An example of Lemma \ref{ldegree} is provided in Figure \ref{degreeexample}, with $\ell=3$ leaks shown. As we can see in the figure, the vertex $v$ can never be colored because all of its neighbors have a leak. Therefore, $v$ must be a part of the $\ell$-forcing set in order to ensure it can be colored.

Lastly, we review the concept of ``forts'' introduced in \cite{compzf}. In Section, \ref{sec.alg}, we implement a constraint generation method to calculate the $\ell$-forcing number similar to \cite{compzf} using the concept of forts in graphs. For the original concept of zero forcing, the definition found in \cite{compzf} states that a (zero forcing) \emph{fort} is a subset $T$ of the vertex set $V$ such that $T$ is non-empty and that no vertex in $V \backslash T$ is adjacent to exactly one vertex in $T$. This means that a fort can also be thought of as a set $T \subset V$ such that $V \backslash T$ is unable to color the entire graph; or in other words, it is a set of uncolored vertices after iteratively applying the color-change rule. We use this understanding of forts to iteratively update a constraint in our integer program for $\ell$-forcing.

\begin{proposition}[Brimkov-Fast-Hicks \cite{compzf}] \label{prop.forts1}
For a graph $G$, a set of vertices $S$ is a zero forcing set if and only if $S$ intersects every fort.
\end{proposition}

\begin{definition} ($\ell$-Forcing Fort).

A set of vertices $T \subset V$ in the graph $G$ is an \emph{($\ell$-forcing) fort} (or just {\emph fort}, if the context is clear), if there are subsets of vertices $S, L \subset V$ such that when $S$ is initially colored and the vertices $L$ have leaks, the set of uncolored vertices after iteratively applying the color-change rule as much as possible is precisely $T$.
\end{definition}

\begin{proposition}\label{prop.fort}
For a graph $G$, a set of vertices $S$ is an $\ell$-forcing set if and only if $S$ intersects every $\ell$-forcing fort.
\end{proposition}

\begin{proof}
Suppose $S$ intersects every $\ell$-forcing fort but is not an $\ell$-forcing set, leaving $T$ uncolored. Then, by definition, $T$ is a fort, but $S$ did not intersect $T$, a contradiction.

Conversely, suppose $S$ is an $\ell$-forcing set but does not intersect every $\ell$-forcing fort. Let $T$ be an $\ell$-forcing fort with $S \cap T = \emptyset$ achieved with leaks $L$. Note that the complement of $T$, $V \setminus T$, has $V \setminus T \supseteq S$. Let $cl(\cdot)$ denote the maximal extent of colored vertices after iteratively applying the color-change rule. Note that the color-change rule does not enable additional vertices to be colored by uncoloring others; hence, $cl(S) \subseteq cl(V \setminus T)$. However, since $T$ is a fort, $cl(V \setminus T)$ is not all of $V$, contradicting that $S$ is an $\ell$-forcing set.
\end{proof}

\section{Leaky Forcing For Certain Families of Graphs} \label{sec.families}

\subsection{Paths}

\begin{proposition} \label{pathleakprop} For the path on $n$ vertices, $P_n$,
\[Z_{(\ell)}(P_n) = \left\{ 
\begin{array}{cc}
1 & \text{ if } \ell = 0 \\
2 & \text{ if } \ell = 1 \\
n & \text{ if } \ell \ge 2 \\
\end{array}\right.\]
\end{proposition}

\begin{proof}
If there are zero leaks on a path it is sufficient to color one endpoint, which ensures that the entire graph may be colored. By Lemma \ref{ldegree}, if there is one leak then all vertices with a degree of one must be colored, which includes the two endpoints of the path. Only coloring the two endpoints is sufficient, since each of them will be able to force adjacent vertices until the vertex with the leak is reached. At this point, both vertices adjacent to the vertex with the leak are colored, so both that vertex and the leak may also be colored. For two or more leaks, we know by Lemma \ref{ldegree} that all vertices of the path must be colored to ensure that the entire graph can be colored.
\end{proof}

\subsection{Trees}

Lemma \ref{minusoneleaf} is provided for use in the proof of Proposition \ref{oneleaktree}.

\begin{lemma}
\label{minusoneleaf}
For a tree $T$ with $t \geq 2$ leaves, if $t - 1$ leaves are colored then they form a 0-forcing set for $T$ (an $\ell$-forcing set with $\ell = 0$ leaks).
\end{lemma}

\begin{proof}

This will be a proof by induction on the number of levels in $T$. Using a predetermined and arbitrary parent vertex $b$ in $T$, level $k$ will be defined as all of the vertices $k$ steps away from $b$.

First, for a tree graph with one level and $t$ leaves, an example of which is shown in Figure \ref{treebasecase}, if $t-1$ leaves are colored then any of them can force the only parent vertex $b$. Once $b$ is colored, then the one uncolored leaf is forced, and all of $T$ is colored.

\begin{center}
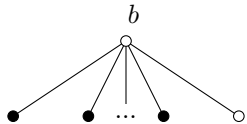
\begin{figure}[!ht]
\begin{tikzpicture}

\draw (0,2) ++(75:0.38cm) node [fill=white,inner sep=4pt](a){$b$};

\draw[thin] (0,2)--(1.5,1);
\draw[thin] (0,2)--(-1.5,1);
\draw[thin] (-0.5,1)--(0,2);
\draw[thin] (0,2)--(0.5,1);

\draw[thin] (0,2)--(0,0.9);
\draw (0,1) ++(000:0cm) node [fill=white,inner sep=4pt](a){...};

\draw[fill=white] (0,2) circle (2pt);
\draw[fill] (-1.5,1) circle (2pt);
\draw[fill] (-0.5,1) circle (2pt);
\draw[fill] (0.5,1) circle (2pt);
\draw[fill=white] (1.5,1) circle (2pt);

\end{tikzpicture}
\caption{A tree graph with one level and $t-1$ leaves colored.}
\label{treebasecase}
\end{figure}
\end{center}

We will assume this to be true for $k$ levels in a tree graph $T$.

For a tree graph on $k+1$ levels with $t$ leaves, assume $t-1$ leaves are colored. Then, when the $t-1$ colored leaves force their parent vertices in layer $k$, the one uncolored leaf will be unable to force. If the uncolored leaf has a colored sibling vertex (a leaf adjacent to the same parent in layer $k$), then its parent in $k$ will be forced by that sibling. If the uncolored leaf has no sibling vertex, then its parent in layer $k$ will not be forced. So, at most, there will be one uncolored leaf vertex on layer $k$, if layer $k+1$ is omitted once it has forced layer $k$.

As a result, if $t-1$ leaves are colored in a tree graph $T$ with $t \geq 2$ leaves, then they form a 0-forcing set for $T$.

\end{proof}

\begin{proposition} \label{oneleaktree}

For a tree graph $T$ with $t$ leaves (except for $T = K_2 = P_2$), $$Z_{(1)}(T) = t.$$  Further, the optimum 1-forcing set exactly contains exactly the leaves of $T$.

\end{proposition}

\begin{proof}

This will be a proof by induction on the number of levels of $T$.

For our base case, consider a tree with one parent vertex $b$ and $t$ leaves in the first level, all of which are colored.

If one leak is added to this graph, there are the following two cases: either it is added to a leaf vertex or it is added to the parent vertex $b$.

\emph{Case 1 - The leak is added to a leaf vertex.}
If the leak is added to a leaf vertex $v$, then $v$ will be unable to force. However, by Lemma \ref{minusoneleaf}, the other $t - 1$ leaves form a 0-forcing set for $T$ without the leak. Once all of $T$ is colored, then the leak is the only uncolored vertex in $T$ and then it can be colored by $v$.

\emph{Case 2 - The leak is added to the parent vertex.}
If the leak is added to the parent vertex $b$, then every leaf will be able to force upward one level and color $b$. At this point, the only uncolored vertex in $T$ is the leak, which can then be colored by $b$.

We will assume this to be true for a tree $T$ with $k$ levels. For a tree on $k+1$ levels, we have two cases: either the leak is added to a leaf on level $k+1$ or the leak is added to a vertex not on level $k+1$.

\emph{Case 1 - The leak is on level $k+1$.}
If the leak is added to a leaf vertex $v$ on level $k+1$, then $v$ will be unable to force. The other $t - 1$ leaves can force level $k$ and at this point we have a tree with $k$ levels and at most one uncolored leaf. By Lemma \ref{minusoneleaf}, this is sufficient to color this tree on $k$ levels. Then, the only uncolored vertex is the leak, which can be forced by $v$.

\emph{Case 2 - The leak is not on level $k+1$.}
If the leak not added to level $k+1$, then all $t$ leaves will be able to force levels $k, k-1, k-2,\dots$ . They can force unimpeded through the levels until they reach the level, let's assume level $j$, which has the leak. Then, every vertex on level $j$ except the vertex with the leak, vertex $u$, will be able to force level $j-1$. We then have a subgraph that is a tree on $j-1$ levels with at least all leaves but one colored. By Lemma \ref{minusoneleaf}, this subgraph can be colored. Then, the only uncolored vertex in $T$ is the leak, which can be colored by $u$.

\end{proof}

\subsection{Cyles}

\begin{proposition} \label{cycle} For the cycle on $n$ vertices, $C_n$,
\[Z_{(\ell)}(C_n) = \left\{ 
\begin{array}{cc}
2 & \text{ if } \ell = 0, 1 \\
n & \text{ if } \ell \ge 2 \\
\end{array}\right.\]
\end{proposition}

\begin{proof}
If there are zero leaks on the cycle, then we only need to color two adjacent vertices in order to force the entire graph. Both of the initially colored vertices have just one uncolored neighbor, so they can both force around the cycle until the entire graph is colored. This initial coloring of two adjacent vertices also works if there is one leak on the graph. With one leak, we have the following two cases:

\emph{Case 1 - The leak is added to a colored vertex.}
If the leak is on one of the colored vertices, then the other may force around the cycle until it reaches the leaky vertex.

\emph{Case 2 - The leak is added to an uncolored vertex.}
If the leak is on an initially uncolored vertex, then both colored vertices may force until they reach the vertex with the leak. Once both adjacent vertices to the vertex with the leak are colored, then the entire graph may be colored.

For two or more leaks, we know by Lemma \ref{ldegree} that all vertices of the cycle must be colored to ensure that the entire graph can be colored.
\end{proof}

\subsection{Complete Graphs}

\begin{proposition} For the complete graph on $n$ vertices, $K_n$, \[Z_{(\ell)}(K_n) = \left\{ 
\begin{array}{cc}
n-1 & \text{ if } \ell \le n-2 \\
n & \text{ if } \ell \ge n-1 \\
\end{array}\right.\]
\end{proposition}

\begin{proof}
In order to force, a vertex must be colored and have all but one neighbor colored. Since every vertex in $K_n$ has $n-1$ neighbors, the minimum number of initially colored vertices in order for any vertex to force is $n-1$: the vertex that will force and $n-2$ of its neighbors. Once $n-1$ vertices are colored, there are $n-1$ colored vertices all adjacent to one uncolored vertex. For the case in which $\ell \leq n-2$, it is not possible for all of these colored vertices to have a leak. As a result, there will be at least one colored vertex with no leak and therefore only one uncolored neighbor, which once colored, will color the entire graph.

For $n-1$ or more leaks, we know by Lemma \ref{ldegree} that all vertices of $K_n$ must be colored to ensure that the entire graph can be colored.
\end{proof}

\subsection{Wheels}

\begin{proposition} For the wheel graph on $n+1$ vertices, $W_n$, \[Z_{(\ell)}(W_n) = \left\{ 
\begin{array}{cc}
3 & \text{ if } \ell = 0, 1 \\
\lceil{2n/3}\rceil +1  & \text{ if } \ell = 2 \\
n  & \text{ if } n > \ell > 2 \\
n + 1  & \text{ if } \ell = n, n + 1 \\
\end{array}\right.\]
\end{proposition}

\begin{proof}
If the number of leaks is 0 or 1, coloring three vertices is sufficient to ensure the entire graph can be colored. For either of these amounts of leaks, all of $W_n$ can be colored if we initially color the center vertex and two adjacent vertices on the outside of the wheel. When there are zero leaks, both of the colored vertices on the outside of the wheel can force their uncolored neighbors, which can continue forcing unimpeded until all of the graph is colored. When there is one leak, we have the following three cases to consider, as illustrated in Figure \ref{oneleakwheel}:

\emph{Case 1 - The leak is on the center vertex.} In this case, the forcing pattern can continue in the same manner as when there are no leaks, as the center vertex does not need to force any of the vertices on the outside of the wheel.

\emph{Case 2 - The leak is on a colored vertex on the outside of the wheel.} In this case, the initially colored vertex with the leak is unable to force. However, the other colored vertex on the outside of the wheel is able to force its uncolored neighbor, which can then continue forcing all $n$ vertices on the outside of the wheel until the vertex with the leak is reached.

\emph{Case 3 - The leak is on an initially uncolored vertex.} Here, we can begin forcing the graph as if there were zero leaks and continue forcing in this manner from both initially colored vertices on the outside of the wheel. They can both force unimpeded around the outside of the wheel until they both reach the vertex with the leak, at which point all of $W_n$ will be colored.

\begin{center}
\begin{figure}[h]
\centering
\begin{tikzpicture}

\draw[] (0,0) arc (30:180:1);

\draw[] (3,0) arc (30:180:1);

\draw[] (6,0) arc (30:180:1);

\draw[dashed] (0,0) arc (30:-180:1);

\draw[dashed] (3,0) arc (30:-180:1);

\draw[dashed] (6,0) arc (30:-180:1);

\draw (-0.866025,-2) ++(0:0cm) node [fill=white,inner sep=4pt](a){ Case 1 };

\draw[fill] (-0.866025,-0.5) circle (2pt);

\draw[thin] (-0.866025,-0.5)--(-1.86603,-0.5);
\draw[thin] (-0.866025,-0.5)--(-0.866025,0.5);
\draw[thin] (-0.866025,-0.5)--(0.7071-0.866025,0.2071);
\draw[thin] (-0.866025,-0.5)--(-0.7071-0.866025,0.2071);

\draw[thin,dashed] (-0.866025,-0.5)--(0.133975,-0.5);
\draw[thin,dashed] (-0.866025,-0.5)--(-0.7071-0.866025,-1.2071);
\draw[thin,dashed] (-0.866025,-0.5)--(-0.5,-0.8);

\draw[dashed,fill=gray!40] (0.133975,-0.5) circle (2pt);
\draw[dashed,fill=gray!40] (-0.7071-0.866025,-1.2071) circle (2pt);

\draw[fill=white] (-0.5,-0.8) circle (2pt);

\draw[fill=gray!40] (-1.86603,-0.5) circle (2pt);
\draw[fill] (-0.866025,0.5) circle (2pt);
\draw[fill=gray!40] (0.7071-0.866025,0.2071) circle (2pt);
\draw[fill] (-0.7071-0.866025,0.2071) circle (2pt);

\draw[->,dashed] (-0.866025,0.5)  to [out=45,in=90,looseness=1.5] (0.7071-0.866025,0.3071);
\draw[->,dashed] (-0.7071-0.866025,0.2071)  to [out=135,in=150,looseness=1.5] (-1.96603,-0.5);

\draw[->,dashed] (0.7071-0.766025,0.2071)  to [out=15,in=45,looseness=1.5] (0.233975,-0.5);
\draw[->,dashed] (-1.96603,-0.5)  to [out=230,in=190,looseness=1.5] (-0.7071-0.966025,-1.2071);

\draw (-0.866025+3,-2) ++(0:0cm) node [fill=white,inner sep=4pt](a){ Case 2 };

\draw[fill] (-0.866025+3,-0.5) circle (2pt);

\draw[thin] (-0.866025+3,-0.5)--(-1.86603+3,-0.5);
\draw[thin] (-0.866025+3,-0.5)--(-0.866025+3,0.5);
\draw[thin] (-0.866025+3,-0.5)--(0.7071-0.866025+3,0.2071);
\draw[thin] (-0.866025+3,-0.5)--(-0.7071-0.866025+3,0.2071);

\draw[thin,dashed] (-0.866025+3,-0.5)--(0.133975+3,-0.5);
\draw[thin,dashed] (-0.866025+3,-0.5)--(-0.7071-0.866025+3,-1.2071);
\draw[thin,dashed] (-0.866025+3,0.5)--(-0.866025+3.5,0.9);

\draw[dashed,fill=gray!40] (0.133975+3,-0.5) circle (2pt);
\draw[dashed,fill=gray!40] (-0.7071-0.866025+3,-1.2071) circle (2pt);

\draw[fill=white] (-0.866025+3.5,0.9) circle (2pt);

\draw[fill=gray!40] (-1.86603+3,-0.5) circle (2pt);
\draw[fill] (-0.866025+3,0.5) circle (2pt);
\draw[fill=gray!40] (0.7071-0.866025+3,0.2071) circle (2pt);
\draw[fill] (-0.7071-0.866025+3,0.2071) circle (2pt);

\draw[->,dashed] (-0.7071-0.866025+3,0.2071)  to [out=135,in=150,looseness=1.5] (-1.96603+3,-0.5);

\draw[->,dashed] (-1.96603+3,-0.5)  to [out=230,in=190,looseness=1.5] (-0.7071-0.966025+3,-1.2071);

\draw[->,dashed] (0.133975+3.1,-0.5)  to [out=45,in=15,looseness=1.5] (0.7071-0.766025+3,0.2071);

\draw (-0.866025+6,-2) ++(0:0cm) node [fill=white,inner sep=4pt](a){ Case 3 };

\draw[fill] (-0.866025+3+3,-0.5) circle (2pt);

\draw[thin] (-0.866025+3+3,-0.5)--(-1.86603+3+3,-0.5);
\draw[thin] (-0.866025+3+3,-0.5)--(-0.866025+3+3,0.5);
\draw[thin] (-0.866025+3+3,-0.5)--(0.7071-0.866025+3+3,0.2071);
\draw[thin] (-0.866025+3+3,-0.5)--(-0.7071-0.866025+3+3,0.2071);

\draw[thin,dashed] (-0.866025+3+3,-0.5)--(0.133975+3+3,-0.5);
\draw[thin,dashed] (-0.866025+3+3,-0.5)--(-0.7071-0.866025+3+3,-1.2071);
\draw[thin,dashed] (0.133975+3+3,-0.5)--(0.133975+3+3.4,-0.8);

\draw[dashed,fill=gray!40] (0.133975+3+3,-0.5) circle (2pt);
\draw[dashed,fill=gray!40] (-0.7071-0.866025+3+3,-1.2071) circle (2pt);

\draw[fill=white] (0.133975+3+3.4,-0.8) circle (2pt);

\draw[fill=gray!40] (-1.86603+3+3,-0.5) circle (2pt);
\draw[fill] (-0.866025+3+3,0.5) circle (2pt);
\draw[fill=gray!40] (0.7071-0.866025+3+3,0.2071) circle (2pt);
\draw[fill] (-0.7071-0.866025+3+3,0.2071) circle (2pt);

\draw[->,dashed] (-0.866025+6,0.5)  to [out=45,in=90,looseness=1.5] (0.7071-0.866025+6,0.3071);
\draw[->,dashed] (-0.7071-0.866025+6,0.2071)  to [out=135,in=150,looseness=1.5] (-1.96603+6,-0.5);

\draw[->,dashed] (0.7071-0.766025+6,0.2071)  to [out=15,in=45,looseness=1.5] (0.233975+6,-0.5);
\draw[->,dashed] (-1.96603+6,-0.5)  to [out=230,in=190,looseness=1.5] (-0.7071-0.966025+6,-1.2071);

\end{tikzpicture}
\caption{An illustration of the three cases possible with one leak on a wheel graph.}
\label{oneleakwheel}
\end{figure}
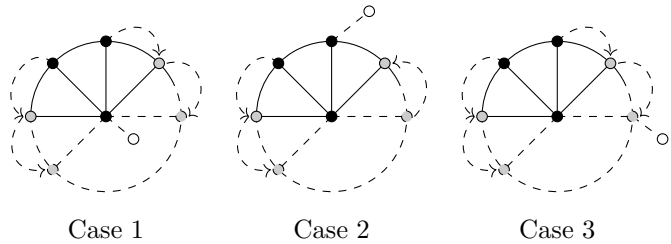
\end{center}

Next, if the number of leaks is 2, we must initially color vertices in a slightly different pattern. As with 0 or 1 leaks, the center vertex and two adjacent vertices on the outside of the wheel must be colored. We must then follow the pattern of leaving one neighboring vertex (to this original pair of colored vertices) uncolored and coloring the next two adjacent vertices for all $n$ vertices on the outside of the graph. If $n$ is not divisible by 3, we still color a pair of vertices in this pattern, even if that results in a path of 3 or 4 vertices colored along the outside of the wheel. This pattern will result in $\lceil{2n/3}\rceil +1$ vertices becoming initially colored. We can not have two adjacent uncolored vertices on the outside of the wheel, since the two leaks could be on their respective colored neighbors on the outside of the wheel, preventing them from becoming colored. This case is shown in Figure \ref{2leakwheel}. If we colored fewer than $\lceil{2n/3}\rceil +1$ vertices, this case would be possible. Therefore, $\lceil{2n/3}\rceil +1$ vertices must be initially colored to ensure the graph can be colored with 2 leaks.

\begin{center}
\begin{figure}[h]
\centering
\begin{tikzpicture}

\draw[thin,dashed] (0,6)--(0.25,6.45);
\draw[thin,dashed] (0.7071+0.45,5-0.7071)--(0.7071,5-0.7071);

\draw [black, thin] (0,5) circle [radius=1 cm];
\draw[fill] (0,5) circle (2pt);

\draw[fill] (-1,5) circle (2pt);
\draw[fill] (0,6) circle (2pt);
\draw[fill] (0,4) circle (2pt);

\draw[fill] (0.7071,5-0.7071) circle (2pt);
\draw[fill] (0-0.7071,5.7071) circle (2pt);
\draw[fill] (0-0.7071,5-0.7071) circle (2pt);
\draw[thin] (0,6)--(0,4);
\draw[thin] (0.7071,5.7071)--(0-0.7071,5-0.7071);
\draw[thin] (0-0.7071,5.7071)--(0.7071,5-0.7071);
\draw[thin] (-1,5)--(1,5);
\draw[fill=white] (1,5) circle (2pt);
\draw[fill=white] (0.7071,5.7071) circle (2pt);

\draw[fill=white] (0.25,6.45) circle (2pt);
\draw[fill=white] (0.7071+0.45,5-0.7071) circle (2pt);

\end{tikzpicture}
\caption{An illustration of a case possible on a wheel with two leaks which prevent the wheel from becoming fully colored.}
\label{2leakwheel}
\end{figure}
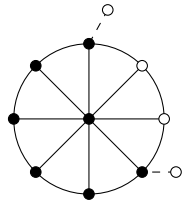
\end{center}

When the number of leaks is greater than 2, we must then color all $n$ vertices on the outside of the wheel (those with degree equal to 3) by Lemma \ref{ldegree}. Coloring these will be sufficient to force the center vertex of the wheel if the number of leaks is less than $n$ because at least one vertex on the outside of the wheel will not have a leak and will therefore be able to force the center vertex.

If the number of leaks is equal to $n$ or $n+1$, we must color all $n+1$ vertices of the wheel graph by Lemma \ref{ldegree} to ensure the graph can be colored, knowing the center vertex has a degree of $n$.

\end{proof}

\section{Cartesian Products} \label{sec.boxproducts}

First, we will define the Cartesian (Box) product of graphs.

\begin{definition} \label{boxproductdef} (Cartesian Product of Graphs).

As found in \cite{MR0209178}, the Cartesian (Box) product of two finite graphs $G$ and $H$, denoted $G \Box H$, is a graph whose vertex set consists of the ordered pairs $(x,y)$, where $x \in V(G)$ and $y \in V(H)$. Two vertices $(x,y)$ and $(x^\prime,y^\prime)$ are adjacent in the Cartesian product if and only if $x = x^\prime$ and $y$ is adjacent to $y^\prime$ in $H$ or $y=y^\prime$ and $x$ is adjacent to $x^\prime$ in $G$.
\end{definition}
For example, the product between a path on 3 vertices and a path on 5 vertices, denoted $P_3 \Box P_5 $, would create a $ 3 \times 5 $ grid of vertices, illustrated in Figure \ref{boxproductexample}.

\begin{center}
\begin{figure}[h]
\centering
\begin{tikzpicture}[scale=0.7]

\draw[step=1cm,black,thin] (0,0) grid (4,2);

\draw [decorate,decoration={brace,amplitude=10pt},xshift=0pt,yshift=0pt] (0,2.5) -- (4,2.5) node [black,midway,xshift=0pt,yshift=20pt] {\footnotesize $P_5$};

\draw [decorate,decoration={brace,amplitude=10pt},xshift=0pt,yshift=0pt] (-0.5,0) -- (-0.5,2) node [black,midway,xshift=-20pt,yshift=0pt] {\footnotesize $P_3$};

\draw[fill=white] (0,1) circle (4pt);
\draw[fill=white] (0,2) circle (4pt);
\draw[fill=white] (0,0) circle (4pt);

\draw[fill=white] (1,1) circle (4pt);
\draw[fill=white] (1,2) circle (4pt);
\draw[fill=white] (1,0) circle (4pt);

\draw[fill=white] (2,0) circle (4pt);
\draw[fill=white] (2,2) circle (4pt);
\draw[fill=white] (2,1) circle (4pt);

\draw[fill=white] (3,2) circle (4pt);
\draw[fill=white] (3,1) circle (4pt);
\draw[fill=white] (3,0) circle (4pt);

\draw[fill=white] (4,0) circle (4pt);
\draw[fill=white] (4,2) circle (4pt);
\draw[fill=white] (4,1) circle (4pt);

\end{tikzpicture}
\caption{The Cartesian product $P_3 \Box P_5$.}
\label{boxproductexample}
\end{figure}
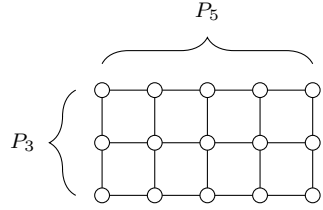
\end{center}


\begin{definition} \label{hypercube}(Hypercube).

The $d$-dimensional hypercube, $Q_d$, is the graph
\[ Q_d := \underbrace{K_2 \Box \ldots \Box K_2}_{n \text{ times}} .\]
\end{definition}
In particular, $Q_1 = K_2$ and $Q_2 = C_4$.

\begin{proposition} \label{boxprodleaks}
For graphs $G$ and $H$ and a number of leaks $\ell \geq 0$, \[ Z_{(\ell)} (G \Box H) \leq min\{|G|Z_{(\ell)}(H),|H|Z_{(\ell)}(G)\}.\]
\end{proposition}

\begin{proof}
Without loss of generality, let us look at $G \Box H$ as $|H|$ copies of the graph $G$ and assume $|H|Z_{(\ell)}(G)\ \leq |G|Z_{(\ell)}(H)$. Choose a minimum $\ell$-forcing set of $G$, $S$. Color all vertices $S \times H$; that is, color all corresponding vertices to $S$ in all $|H|$ copies of $G$. It suffices to show that $S \times H$ is an $\ell$-forcing set of the graph $G \Box H$.

Choose $\ell$ leaks of $G \Box H$: $(g_1, h_1), \ldots, (g_\ell, h_\ell)$. Let $L_G = \{g_1, \ldots, g_\ell \}$. We will show that $S \times H$ is an $\ell$-forcing set for the set of leaks $L_G \times H$.

Let $S = S_0$ be the set of initially colored vertices of $G$, and let $S_i$ be the set of vertices colored after $i$ applications of the color change rule when forcing with the set of leaks $L_G$. Necessarily, each $S_i$ is an $\ell$-forcing set of $G$. Here, when applying the color-change rule, we apply it to all vertices in the graph in each subsequent application in order to color all vertices possible on each iteration.

Observe that any uncolored neighbor of any vertex $(g,h) \in S \times H$, $(a,b)$, must have $b=h$; in other words, any uncolored neighbor of $(g,h)$ must be in the same copy of $G$ as $(g,h)$. Therefore, since $S$ is an $\ell$-forcing set of $G$, all vertices $S_1 \times H$ can become colored after applying one iteration of the color-change rule on $G \Box H$. Inductively, since $S_i$ is an $\ell$-forcing set of $G$, whenever $S_i \times H$ is colored, $S_{i+1} \times H$ will be colored. It follows that all of $G \times H$ will be colored. Hence, $S \times H$ is an $\ell$-forcing set of $G \Box H$ with leaks $L_G \times H$. Since any set of possible leaks is a subset of some $L_G \times H$, this completes the proof.
\end{proof}

\subsection{Leaky Forcing on Hypercubes}

In this subsection, we will compute the $\ell$-forcing number of hypercubes for some values of $\ell$ as well as provide upper bounds in other cases. To begin, let us recall the zero forcing number of the hypercube.

\begin{proposition}[AIM Group \cite{zf_aim}] \label{zf_cube}
For $d \ge 1$,
\[ Z_{(0)}(Q_d) = 2^{d-1}.\]
\end{proposition}


\begin{proposition} For $d \ge 2$,
\[ Z_{ (1) }(Q_d) = 2^{d-1}. \]
\end{proposition}

\begin{proof}
By Proposition \ref{boxprodleaks}, since $Q_d = Q_{d-2} \Box C_4$, $Z_{ (1) }(Q_d) \le 2^{d-2}  Z_{ (1) }(C_4)$. By Proposition \ref{cycle}, $Z_{ (1) }(C_4)=2$, so altogether,  $Z_{ (1) }(Q_d) \le 2^{d-1}$. For equality, notice that $Z_{ (0) }(Q_d) = 2^{d-1} \le Z_{ (1) }(Q_d)$
\end{proof}

\begin{proposition}
\[ Z_{ (2) }(Q_3) = 6. \]
\end{proposition}

Arguably, the easiest proof for the above proposition is ``the computer said so'' (simply try all $2^8$ possible subsets of vertices each with all 28 possible leak placements). However, we present a more elegant proof using a fort intersection argument more in line with Proposition \ref{prop.fort} and how our algorithm works in Section \ref{sec.alg} based on the work in \cite{compzf}.

\begin{proof}
Let us view the cube as a graph whose vertices are binary strings of length 3.

Color all vertices except 000 and 001.
We will show that this is a $2$-forcing set.
If there is no leak on the vertex 100 or no leak on 010, then 000 will be colored which can then force 001. However, if both 100 and 010 have leaks, then 001 will be colored, after which, 001 forces 000.

To see this is optimal, note any set of vertices consisting of all but a pair of vertices $u_1, u_2$, distance 2 apart is a failed $2$-forcing set (as the leaks could be on the vertices representing the complement strings of $u_1$ and $u_2$, preventing them from forcing); hence, $\{u_1, u_2\}$ is a fort, and by Proposition \ref{prop.fort}, for each pair of vertices distance 2 apart, one of them must be initially colored to be 2-forcing set. This is equivalent to finding a traversal (hitting all edges using vertices) of two disjoint copies of $K_4$ which requires 6 vertices.
\end{proof}




\begin{proposition}  \label{exact_cubes}
\begin{eqnarray*} 
Z_{ (2) }(Q_4) &=& 8 \\
Z_{ (3) }(Q_5) &=& 16.
\end{eqnarray*}
\end{proposition}

Unfortunately, we do not have a straightforward proof of the above proposition, and this time, we omit a formal proof and resort to ``the computer said so'' using the computational techniques in Section \ref{sec.alg}.
For $Z_{ (2) }(Q_4)$, one can check that coloring any subcube, $Q_3$, of $Q_4$ is a $2$-forcing set, and this is optimal as $Z_{(0)}(Q_4)=8 \le Z_{ (2) }(Q_4)$ by Propositions \ref{zf_chain} and \ref{zf_cube}. A similar argument applies for $Z_{ (3) }(Q_5)$. \hfill \qed

\begin{proposition} For $d \ge 4$,
\[ Z_{ (2) }(Q_d) = 2^{d-1}. \]
\end{proposition}

\begin{proof}
By Proposition \ref{boxprodleaks}, since $Q_d = Q_{d-4} \Box Q_4$, $Z_{ (2) }(Q_d) \le 2^{d-4}  Z_{ (2) }(Q_4)
= 2^{d-4} \cdot 8 = 2^{d-1}.$  Equality is achieved, as $Z_{ (0) }(Q_d) = 2^{d-1} \le Z_{ (2) }(Q_d)$ by Propositions \ref{zf_chain} and \ref{zf_cube}.
\end{proof}



\begin{proposition}
\[ Z_{ (3) }(Q_4) = 10 \]
\end{proposition}

This follows using the computational techniques in Section \ref{sec.alg}. \hfill \qed

An optimal 3-forcing set on 10 vertices can be seen in Figure \ref{fig.cube34}.

\begin{center}
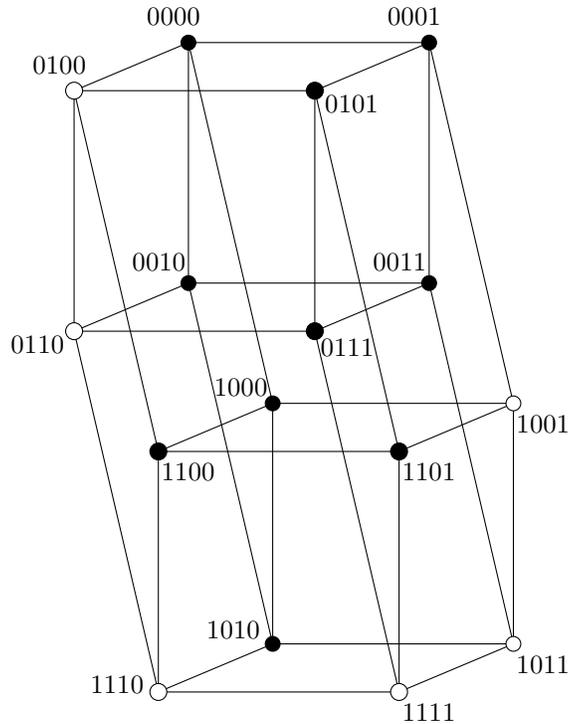
\begin{figure}[!h]
\centering
\begin{tikzpicture}[scale=0.8]


\draw (1.9,0.8) ++(120:0.5cm) node [fill=white,inner sep=4pt](a){0000};
\draw (5.9,0.8) ++(120:0.5cm) node [fill=white,inner sep=4pt](a){0001};
\draw (4,0) ++(340:0.65cm) node [fill=white,inner sep=4pt](a){0101};
\draw (0,0) ++(120:0.5cm) node [fill=white,inner sep=4pt](a){0100};
\draw (4,-4) ++(333:0.6cm) node [fill=white,inner sep=4pt](a){0111};
\draw (1.9,-3.2) ++(145:0.6cm) node [fill=white,inner sep=4pt](a){0010};
\draw (0,-4) ++(200:0.65cm) node [fill=white,inner sep=4pt](a){0110};
\draw (5.9,-3.2) ++(145:0.6cm) node [fill=white,inner sep=4pt](a){0011};

\draw[thin] (0,0)--(0,-4)--(4,-4)--(4,0)--(0,0);
\draw[thin] (1.9,0.8)--(5.9,0.8)--(5.9,-3.2)--(1.9,-3.2)--(1.9,0.8);

\draw[thin] (0,0)--(1.9,0.8);
\draw[thin] (4,0)--(5.9,0.8);
\draw[thin] (0,-4)--(1.9,-3.2);
\draw[thin] (4,-4)--(5.9,-3.2);

\draw (1.4,-6) ++(325:0.6cm) node [fill=white,inner sep=4pt](a){1100};
\draw (1.4,-10) ++(170:0.7cm) node [fill=white,inner sep=4pt](a){1110};
\draw (5.4,-6) ++(325:0.6cm) node [fill=white,inner sep=4pt](a){1101};
\draw (5.4,-10) ++(325:0.6cm) node [fill=white,inner sep=4pt](a){1111};
\draw (3.3,-5.2) ++(153:0.6cm) node [fill=white,inner sep=4pt](a){1000};
\draw (7.3,-5.2) ++(325:0.6cm) node [fill=white,inner sep=4pt](a){1001};
\draw (3.3,-9.2) ++(160:0.7cm) node [fill=white,inner sep=4pt](a){1010};
\draw (7.3,-9.2) ++(325:0.6cm) node [fill=white,inner sep=4pt](a){1011};

\draw[thin] (1.4,-6)--(1.4,-10)--(5.4,-10)--(5.4,-6)--(1.4,-6);
\draw[thin] (3.3,-5.2)--(7.3,-5.2)--(7.3,-9.2)--(3.3,-9.2)--(3.3,-5.2);

\draw[thin] (1.4,-6)--(3.3,-5.2);
\draw[thin] (5.4,-10)--(7.3,-9.2);
\draw[thin] (1.4,-10)--(3.3,-9.2);
\draw[thin] (5.4,-6)--(7.3,-5.2);

\draw[thin] (1.4,-6)--(0,0);
\draw[thin] (5.4,-10)--(4,-4);
\draw[thin] (1.4,-10)--(0,-4);
\draw[thin] (5.4,-6)--(4,0);

\draw[thin] (3.3,-5.2)--(1.9,0.8);
\draw[thin] (3.3,-9.2)--(1.9,-3.2);
\draw[thin] (7.3,-5.2)--(5.9,0.8);
\draw[thin] (7.3,-9.2)--(5.9,-3.2);

\draw[fill,fill=black] (1.9,0.8) circle (3.5pt);
\draw[fill,fill=black] (5.9,0.8) circle (3.5pt);
\draw[fill,fill=black] (1.9,-3.2) circle (3.5pt);
\draw[fill,fill=black] (3.3,-5.2) circle (3.5pt);
\draw[fill,fill=white] (0,0) circle (4pt);
\draw[fill,fill=black] (5.9,-3.2) circle (3.5pt);
\draw[fill,fill=white] (7.3,-5.2) circle (3.5pt);
\draw[fill,fill=black] (1.4,-6) circle (4pt);
\draw[fill,fill=white] (0,-4) circle (4pt);
\draw[fill,fill=black] (3.3,-9.2) circle (3.5pt);
\draw[fill,fill=black] (4,0) circle (4pt);
\draw[fill,fill=white] (7.3,-9.2) circle (3.5pt);
\draw[fill,fill=white] (1.4,-10) circle (4pt);
\draw[fill,fill=black] (5.4,-6) circle (4pt);
\draw[fill,fill=black] (4,-4) circle (4pt);
\draw[fill,fill=white] (5.4,-10) circle (4pt);

\end{tikzpicture}
\caption{An optimal $3$-forcing set of $Q_4$.}
\label{fig.cube34}
\end{figure}
\end{center}






\begin{proposition} For $d \ge 5$,
\[ Z_{ (3) }(Q_d) = 2^{d-1}. \]
\end{proposition}

\begin{proof}
By Proposition \ref{boxprodleaks}, since $Q_d = Q_{d-5} \Box Q_5$, $Z_{ (3) }(Q_d) \le 2^{d-5} Z_{ (3) }(Q_5)
= 2^{d-5} \cdot 16 = 2^{d-1}.$ As before, equality follows by applying Propositions \ref{zf_chain} and \ref{zf_cube}.
\end{proof}


\def\arraystretch{1.5}
\begin{center}
\begin{table} [h] \label{tab.cube}
    \begin{tabular}{ c  c || c | c | c | c |}
    \cline{3-6}
     & & \multicolumn{4}{|c|}{$d$} \\
    \cline{3-6}
      &  & $3$ & $4$ & $5$ & $6$\\ \hline \hline

    \multicolumn{1}{|c|}{} & $0$ & $4$ & $8$ & $16$ & $32$ \\ \cline{2-6}

    \multicolumn{1}{|c|}{ \multirow{2}{1em}{$\ell$} } & $1$ & $4$ & $8$ & $16$ & $32$ \\
    \cline{2-6}
    \multicolumn{1}{|c|}{} & $2$ & $6$ & $8$ & $16$ & $32$\\
    \cline{2-6}
    \multicolumn{1}{|c|}{} & $3$ & $8$ & $10$ & $16$ & $32$ \\
    \cline{2-6}
    \multicolumn{1}{|c|}{} & $4$ & $\vdots$ & $16$ & ? & ? \\
   \hline
    \end{tabular}
\caption{ Summary of values of $Z_{(\ell)}(Q_d)$.}\label{tab.cubes}
\end{table}
\end{center}

A summary of the known values of $Z_{(\ell)}(Q_d)$ can be found in Table \ref{tab.cubes}. While it appears that $Z_{(d-2)}(Q_d) = 2^{d-1}$ (and therefore $Z_{(j)}(Q_d) = 2^{d-1}$ for any $j \le d-2$), we are hesitant to make such a conjecture. Also, the nontrivial values of $Z_{(d-1)}(Q_d)$ do not appear to be powers of 2; however, they do not appear to be following any specific pattern.


\subsection{Leaky Forcing on Grid Graphs, $P_n \Box P_m$}

Before we begin leaky forcing on grid graphs, let us recall the zero forcing number of grid graphs.

\begin{proposition}[AIM Group \cite{zf_aim}] \label{zf_array prop}
For $1 \le n \le m$,
\[ Z_{(0)}(P_n \Box P_m) = n.\]
\end{proposition}

As a starting point, Proposition \ref{boxprodleaks} gives us the following upper bound for 1-forcing.

\begin{proposition}
For $1 \le n \le m$,
\[ Z_{(1)}(P_n \Box P_m) \le 2n.\]
\end{proposition}

In this section, we will prove various other upper bounds on the the 1-forcing number for grid graphs. Our most general bound which applies to all grid graphs is the following:

\begin{theorem} \label{thm.uppergrid}
For the box product of paths, $P_n \Box P_m$, with $m \ge n$ \[Z_{(1)}(P_n \Box P_m) \leq 2m-n. \]
\end{theorem}

Before proving Theorem \ref{thm.uppergrid}, we need to establish some definitions which we will use in the proof.

The proof of Theorem \ref{thm.uppergrid} follows immediately from  Proposition \ref{array pattern prop} which we prove below.

\begin{definition} \label{coordinate def}(Coordinate System for $P_n \Box P_m$ Graphs).

In describing these $P_n \Box P_m$ graphs, we will use a coordinate system in which each vertex position is described by $(x,y)$. Here, $x$ describes the vertical position and $y$ describes the horizontal position relative to the vertex in the top left corner of the $P_n \Box P_m$ graph at position $(1,1)$. An example of this coordinate system is shown in Figure \ref{array coordinates} on a $P_7 \Box P_7$ graph.

\end{definition}

\begin{center}
\begin{figure}[h]
\centering
\begin{tikzpicture}[scale=0.7]

\draw[step=1cm,black,thin] (0,0) grid (6,6);

\draw [decorate,decoration={brace,amplitude=10pt},xshift=0pt,yshift=0pt] (0,7.5) -- (6,7.5) node [black,midway,xshift=0pt,yshift=15pt] {\footnotesize $m$};

\draw [decorate,decoration={brace,amplitude=10pt},xshift=0pt,yshift=0pt] (-1.5,0) -- (-1.5,6) node [black,midway,xshift=-17pt,yshift=0pt] {\footnotesize $n$};

\draw (-1,6) ++(180:0cm) node [fill=white,inner sep=2.3pt](a){$1$};
\draw (-1,5) ++(180:0cm) node [fill=white,inner sep=2.3pt](a){$2$};
\draw (-1,4) ++(180:0cm) node [fill=white,inner sep=2.3pt](a){$3$};
\draw (-1,3) ++(180:0cm) node [fill=white,inner sep=2.3pt](a){$4$};
\draw (-1,2) ++(180:0cm) node [fill=white,inner sep=2.3pt](a){$5$};
\draw (-1,1) ++(180:0cm) node [fill=white,inner sep=2.3pt](a){$6$};
\draw (-1,0) ++(180:0cm) node [fill=white,inner sep=2.3pt](a){$7$};

\draw (0,7) ++(180:0cm) node [fill=white,inner sep=2.3pt](a){$1$};
\draw (1,7) ++(180:0cm) node [fill=white,inner sep=2.3pt](a){$2$};
\draw (2,7) ++(180:0cm) node [fill=white,inner sep=2.3pt](a){$3$};
\draw (3,7) ++(180:0cm) node [fill=white,inner sep=2.3pt](a){$4$};
\draw (4,7) ++(180:0cm) node [fill=white,inner sep=2.3pt](a){$5$};
\draw (5,7) ++(180:0cm) node [fill=white,inner sep=2.3pt](a){$6$};
\draw (6,7) ++(180:0cm) node [fill=white,inner sep=2.3pt](a){$7$};

\draw (0,6.4) ++(180:0cm) node [fill=white,inner sep=2.3pt](a){\footnotesize $(1,1)$};

\draw (6,6.4) ++(180:0cm) node [fill=white,inner sep=2.3pt](a){\footnotesize $(1,7)$};

\draw (0,-0.5) ++(180:0cm) node [fill=white,inner sep=2.3pt](a){\footnotesize $(7,1)$};

\draw (6,-0.5) ++(180:0cm) node [fill=white,inner sep=2.3pt](a){\footnotesize $(7,7)$};

\draw[fill=white] (3,6) circle (4pt);
\draw[fill=white] (3,5) circle (4pt);
\draw[fill=white] (3,4) circle (4pt);
\draw[fill=white] (3,3) circle (4pt);
\draw[fill=white] (2,4) circle (4pt);
\draw[fill=white] (2,6) circle (4pt);
\draw[fill=white] (4,5) circle (4pt);

\draw[fill=white] (0,0) circle (4pt);
\draw[fill=white] (0,1) circle (4pt);
\draw[fill=white] (0,2) circle (4pt);
\draw[fill=white] (0,3) circle (4pt);
\draw[fill=white] (0,4) circle (4pt);
\draw[fill=white] (0,5) circle (4pt);
\draw[fill=white] (0,6) circle (4pt);

\draw[fill=white] (1,0) circle (4pt);
\draw[fill=white] (1,1) circle (4pt);
\draw[fill=white] (1,2) circle (4pt);
\draw[fill=white] (1,3) circle (4pt);
\draw[fill=white] (1,4) circle (4pt);
\draw[fill=white] (1,5) circle (4pt);
\draw[fill=white] (1,6) circle (4pt);

\draw[fill=white] (5,0) circle (4pt);
\draw[fill=white] (5,1) circle (4pt);
\draw[fill=white] (5,2) circle (4pt);
\draw[fill=white] (5,3) circle (4pt);
\draw[fill=white] (5,4) circle (4pt);
\draw[fill=white] (5,5) circle (4pt);
\draw[fill=white] (5,6) circle (4pt);

\draw[fill=white] (6,0) circle (4pt);
\draw[fill=white] (6,1) circle (4pt);
\draw[fill=white] (6,2) circle (4pt);
\draw[fill=white] (6,3) circle (4pt);
\draw[fill=white] (6,4) circle (4pt);
\draw[fill=white] (6,5) circle (4pt);
\draw[fill=white] (6,6) circle (4pt);

\draw[fill=white] (2,0) circle (4pt);
\draw[fill=white] (2,5) circle (4pt);
\draw[fill=white] (2,3) circle (4pt);
\draw[fill=white] (2,2) circle (4pt);
\draw[fill=white] (2,1) circle (4pt);

\draw[fill=white] (3,0) circle (4pt);
\draw[fill=white] (3,2) circle (4pt);
\draw[fill=white] (3,1) circle (4pt);

\draw[fill=white] (4,6) circle (4pt);
\draw[fill=white] (4,4) circle (4pt);
\draw[fill=white] (4,3) circle (4pt);
\draw[fill=white] (4,2) circle (4pt);
\draw[fill=white] (4,1) circle (4pt);
\draw[fill=white] (4,0) circle (4pt);

\end{tikzpicture}
\caption{An illustration of the coordinate system used, shown here with $n=m=7$.}
\label{array coordinates}
\end{figure}
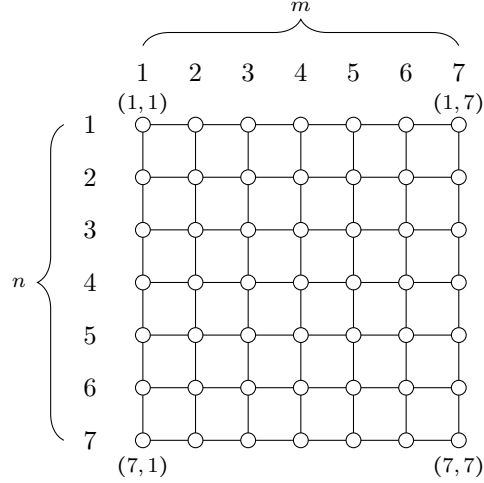
\end{center}

\begin{definition} \label{array pattern}(Initial pattern for $P_n \Box P_m$ graphs).

In the case in which $m>n$, we must follow one of the patterns below. The slight variation is for whether our $n$ value is odd or even. If $m=n$, the pattern remains the same on the left side of the grid, while the colored vertices on the right hand side due to the difference between $m$ and $n$ are omitted.

\emph{Case 1 - $n$ is odd.} If $n$ is odd, we must color vertices as follows, an example of which is provided in Figure \ref{odd n array}. We will use the same coordinate system as described in Definition \ref{coordinate def}. To begin, vertices at the coordinates
$$\left(\Bigl\lceil \frac{n}{2} \Bigr\rceil, 1 \right), \left(\Bigl\lceil \frac{n}{2} \Bigr\rceil, 2 \right), \left(\Bigl\lceil \frac{n}{2} \Bigr\rceil, 3 \right), \ldots , \left(\Bigl\lceil \frac{n}{2} \Bigr\rceil,\Bigl\lceil \frac{n}{2} \Bigr\rceil \right)$$ must be colored. In addition, either the vertices described by $$\left(\Bigl\lceil \frac{n}{2} \Bigr\rceil -1, 1 \right), \left(\Bigl\lceil \frac{n}{2} \Bigr\rceil -1,2 \right), \left(\Bigl\lceil \frac{n}{2} \Bigr\rceil -1, 3 \right), \ldots ,\left(\Bigl\lceil \frac{n}{2} \Bigr\rceil -1,\Bigl\lceil \frac{n}{2} \Bigr\rceil - 1 \right)$$ must be colored, or the vertices described by $$\left(\Bigl\lceil \frac{n}{2} \Bigr\rceil + 1, 1 \right), \left(\Bigl\lceil \frac{n}{2} \Bigr\rceil + 1,2 \right), \left(\Bigl\lceil \frac{n}{2} \Bigr\rceil + 1, 3 \right), \ldots ,\left(\Bigl\lceil \frac{n}{2} \Bigr\rceil + 1,\Bigl\lceil \frac{n}{2} \Bigr\rceil - 1 \right)$$ must be colored. There will be $n$ vertices colored in this portion of the coloring.

To finish the pattern, vertices at the coordinates
$$(n+1, 1), (n + 2,1), (n + 3, 1), \ldots ,(m,1)$$
must be colored, as well as those at the coordinates
$$(n+1, n), (n + 2,n), (n + 3, n), \ldots ,(m,n).$$

Once all vertices in this pattern are colored, there will be $2m-n$ initially colored vertices.

\begin{center}
\begin{figure}[h]
\centering
\begin{tikzpicture}[scale=0.7]

\draw[step=1cm,black,thin] (0,0) grid (9,6);

\draw [decorate,decoration={brace,amplitude=10pt},xshift=0pt,yshift=0pt] (0,7.5) -- (9,7.5) node [black,midway,xshift=0pt,yshift=15pt] {\footnotesize $m$};

\draw [decorate,decoration={brace,amplitude=10pt},xshift=0pt,yshift=0pt] (7,6.5) -- (9,6.5) node [black,midway,xshift=0pt,yshift=15pt] {\footnotesize $m-n$};

\draw[fill=white] (3,6) circle (4pt);
\draw[fill=white] (3,5) circle (4pt);
\draw[fill=white] (3,4) circle (4pt);
\draw[fill] (3,3) circle (4pt);
\draw[fill] (2,4) circle (4pt);
\draw[fill=white] (2,6) circle (4pt);
\draw[fill=white] (4,5) circle (4pt);

\draw[fill=white] (0,1) circle (4pt);
\draw[fill=white] (0,2) circle (4pt);
\draw[fill] (0,3) circle (4pt);
\draw[fill] (0,4) circle (4pt);
\draw[fill=white] (0,5) circle (4pt);
\draw[fill=white] (0,6) circle (4pt);

\draw[fill=white] (1,1) circle (4pt);
\draw[fill=white] (1,2) circle (4pt);
\draw[fill] (1,3) circle (4pt);
\draw[fill] (1,4) circle (4pt);
\draw[fill=white] (1,5) circle (4pt);
\draw[fill=white] (1,6) circle (4pt);

\draw[fill=white] (5,1) circle (4pt);
\draw[fill=white] (5,2) circle (4pt);
\draw[fill=white] (5,3) circle (4pt);
\draw[fill=white] (5,4) circle (4pt);
\draw[fill=white] (5,5) circle (4pt);
\draw[fill=white] (5,6) circle (4pt);

\draw[fill=white] (6,1) circle (4pt);
\draw[fill=white] (6,2) circle (4pt);
\draw[fill=white] (6,3) circle (4pt);
\draw[fill=white] (6,4) circle (4pt);
\draw[fill=white] (6,5) circle (4pt);
\draw[fill=white] (6,6) circle (4pt);

\draw[fill=white] (7,1) circle (4pt);
\draw[fill=white] (7,2) circle (4pt);
\draw[fill=white] (7,3) circle (4pt);
\draw[fill=white] (7,4) circle (4pt);
\draw[fill=white] (7,5) circle (4pt);
\draw[fill] (7,6) circle (4pt);

\draw[fill=white] (8,1) circle (4pt);
\draw[fill=white] (8,2) circle (4pt);
\draw[fill=white] (8,3) circle (4pt);
\draw[fill=white] (8,4) circle (4pt);
\draw[fill=white] (8,5) circle (4pt);
\draw[fill] (8,6) circle (4pt);

\draw[fill=white] (9,1) circle (4pt);
\draw[fill=white] (9,2) circle (4pt);
\draw[fill=white] (9,3) circle (4pt);
\draw[fill=white] (9,4) circle (4pt);
\draw[fill=white] (9,5) circle (4pt);
\draw[fill] (9,6) circle (4pt);

\draw[fill=white] (0,0) circle (4pt);
\draw[fill=white] (1,0) circle (4pt);
\draw[fill=white] (2,0) circle (4pt);
\draw[fill=white] (3,0) circle (4pt);
\draw[fill=white] (4,0) circle (4pt);
\draw[fill=white] (5,0) circle (4pt);
\draw[fill=white] (6,0) circle (4pt);
\draw[fill] (7,0) circle (4pt);
\draw[fill] (8,0) circle (4pt);
\draw[fill] (9,0) circle (4pt);

\draw[fill=white] (2,5) circle (4pt);
\draw[fill] (2,3) circle (4pt);
\draw[fill=white] (2,2) circle (4pt);
\draw[fill=white] (2,1) circle (4pt);

\draw[fill=white] (3,2) circle (4pt);
\draw[fill=white] (3,1) circle (4pt);

\draw[fill=white] (4,6) circle (4pt);
\draw[fill=white] (4,4) circle (4pt);
\draw[fill=white] (4,3) circle (4pt);
\draw[fill=white] (4,2) circle (4pt);
\draw[fill=white] (4,1) circle (4pt);

\draw [decorate,decoration={brace,amplitude=10pt},xshift=0pt,yshift=0pt] (-1.8,0) -- (-1.8, 6) node [black,midway,xshift=-17pt,yshift=0pt] {\footnotesize $n$};

\draw [decorate,decoration={brace,amplitude=10pt},xshift=0pt,yshift=0pt] (-0.6,3) -- (-0.6, 6) node [black,midway,xshift=-19pt,yshift=0pt] {\footnotesize $\lceil \frac{n}{2} \rceil$};

\end{tikzpicture}
\caption{Pattern for $n$ odd, shown here on a $P_{7} \Box P_{10}$ graph}
\label{odd n array}
\end{figure}
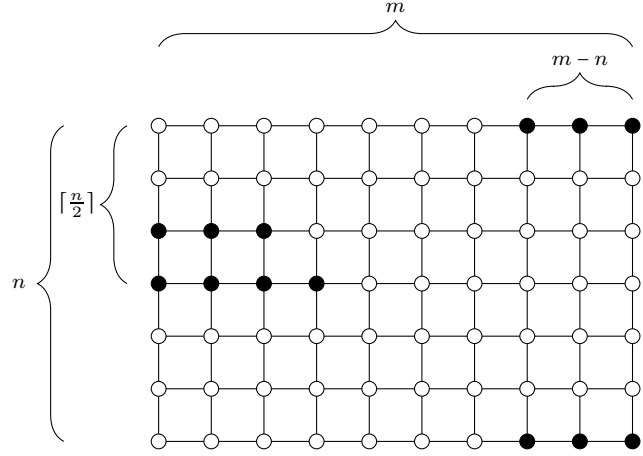
\end{center}

\emph{Case 2 - $n$ is even.} If $n$ is even, we must color vertices as follows, an example of which is provided in Figure \ref{even n array}. We will use the same coordinate system as described in Definition \ref{coordinate def}. To begin, vertices at the coordinates
$$\left(\frac{n}{2}, 1\right), \left(\frac{n}{2}, 2\right), \left(\frac{n}{2}, 3\right), \ldots , \left(\frac{n}{2},\frac{n}{2}\right)$$ and $$\left(\frac{n}{2}+1, 1\right), \left(\frac{n}{2}+1, 2\right), \left(\frac{n}{2}+1, 3\right), \ldots , \left(\frac{n}{2}+1,\frac{n}{2}\right)$$ must be colored. There will be $n$ vertices colored in this portion of the coloring.

To finish the pattern, vertices at the coordinates
$$(n+1, 1), (n + 2,1), (n + 3, 1), \ldots ,(m,1)$$
must be colored, as well as those at the coordinates
$$(n+1, n), (n + 2,n), (n + 3, n), \ldots ,(m,n).$$

Once all vertices in this pattern are colored, there will be $2m-n$ initially colored vertices.

\begin{center}
\begin{figure}[h]
\centering
\begin{tikzpicture}[scale=0.7]

\draw[step=1cm,black,thin] (0,1) grid (9,6);

\draw [decorate,decoration={brace,amplitude=10pt},xshift=0pt,yshift=0pt] (0,7.5) -- (9,7.5) node [black,midway,xshift=0pt,yshift=15pt] {\footnotesize $m$};

\draw [decorate,decoration={brace,amplitude=10pt},xshift=0pt,yshift=0pt] (6,6.5) -- (9,6.5) node [black,midway,xshift=0pt,yshift=15pt] {\footnotesize $m-n$};

\draw[fill=white] (3,6) circle (4pt);
\draw[fill=white] (3,5) circle (4pt);
\draw[fill=white] (3,4) circle (4pt);
\draw[fill=white] (3,3) circle (4pt);
\draw[fill] (2,4) circle (4pt);
\draw[fill=white] (2,6) circle (4pt);
\draw[fill=white] (4,5) circle (4pt);

\draw[fill=white] (0,1) circle (4pt);
\draw[fill=white] (0,2) circle (4pt);
\draw[fill] (0,3) circle (4pt);
\draw[fill] (0,4) circle (4pt);
\draw[fill=white] (0,5) circle (4pt);
\draw[fill=white] (0,6) circle (4pt);

\draw[fill=white] (1,1) circle (4pt);
\draw[fill=white] (1,2) circle (4pt);
\draw[fill] (1,3) circle (4pt);
\draw[fill] (1,4) circle (4pt);
\draw[fill=white] (1,5) circle (4pt);
\draw[fill=white] (1,6) circle (4pt);

\draw[fill=white] (5,1) circle (4pt);
\draw[fill=white] (5,2) circle (4pt);
\draw[fill=white] (5,3) circle (4pt);
\draw[fill=white] (5,4) circle (4pt);
\draw[fill=white] (5,5) circle (4pt);
\draw[fill=white] (5,6) circle (4pt);

\draw[fill] (6,1) circle (4pt);
\draw[fill=white] (6,2) circle (4pt);
\draw[fill=white] (6,3) circle (4pt);
\draw[fill=white] (6,4) circle (4pt);
\draw[fill=white] (6,5) circle (4pt);
\draw[fill] (6,6) circle (4pt);

\draw[fill] (7,1) circle (4pt);
\draw[fill=white] (7,2) circle (4pt);
\draw[fill=white] (7,3) circle (4pt);
\draw[fill=white] (7,4) circle (4pt);
\draw[fill=white] (7,5) circle (4pt);
\draw[fill] (7,6) circle (4pt);

\draw[fill] (8,1) circle (4pt);
\draw[fill=white] (8,2) circle (4pt);
\draw[fill=white] (8,3) circle (4pt);
\draw[fill=white] (8,4) circle (4pt);
\draw[fill=white] (8,5) circle (4pt);
\draw[fill] (8,6) circle (4pt);

\draw[fill] (9,1) circle (4pt);
\draw[fill=white] (9,2) circle (4pt);
\draw[fill=white] (9,3) circle (4pt);
\draw[fill=white] (9,4) circle (4pt);
\draw[fill=white] (9,5) circle (4pt);
\draw[fill] (9,6) circle (4pt);

\draw[fill=white] (2,5) circle (4pt);
\draw[fill] (2,3) circle (4pt);
\draw[fill=white] (2,2) circle (4pt);
\draw[fill=white] (2,1) circle (4pt);

\draw[fill=white] (3,2) circle (4pt);
\draw[fill=white] (3,1) circle (4pt);

\draw[fill=white] (4,6) circle (4pt);
\draw[fill=white] (4,4) circle (4pt);
\draw[fill=white] (4,3) circle (4pt);
\draw[fill=white] (4,2) circle (4pt);
\draw[fill=white] (4,1) circle (4pt);

\draw [decorate,decoration={brace,amplitude=10pt},xshift=0pt,yshift=0pt] (-1.8,1) -- (-1.8, 6) node [black,midway,xshift=-17pt,yshift=0pt] {\footnotesize $n$};

\draw [decorate,decoration={brace,amplitude=10pt},xshift=0pt,yshift=0pt] (-0.6,4) -- (-0.6, 6) node [black,midway,xshift=-19pt,yshift=0pt] {\footnotesize $\frac{n}{2}$};

\end{tikzpicture}
\caption{Pattern for $n$ even, shown here on a $P_{6} \Box P_{10}$ graph.}
\label{even n array}
\end{figure}
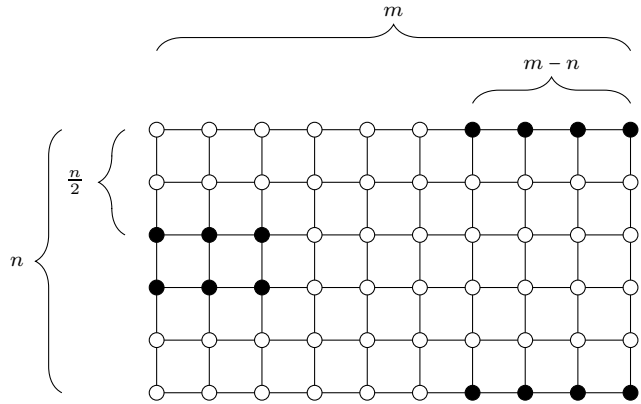
\end{center}

\end{definition}

We will find the following lemma useful throughout this section.

\begin{lemma} \label{lem.sides}
For the grid $P_n \Box P_m$, the set of vertices along any two sides is a 1-forcing set.
\end{lemma}

\begin{proof}
If the sides are nonconsecutive, then the proof follows from the proof of Proposition \ref{boxprodleaks}. If the sides are consecutive, use the vertices on the row and column until the row or column containing the leak is colored (if the leak is a colored vertex, that side (or sides) will not force). Observe that necessarily, all vertices adjacent to the leak are now colored. Hence, we can continue forcing in any one direction to color all of the vertices.
\end{proof}

\begin{proposition} \label{array pattern prop}
The initial coloring described in Definition \ref{array pattern} forms a $1$-forcing set for the graph $P_{n} \Box P_{m}$.
\end{proposition}

\begin{proof}
To begin, we will first visualize the $P_{n} \Box P_{m}$ grid as split into four quadrants, an example of which is given in Figure \ref{subregions array}. A vertical cut is pictured halfway down the side of length $m$ and a horizontal cut is similarly pictured halfway down the side of length $n$. We will then have the following cases:

\emph{Case 1 - The leak is in quadrant $II$ or $III$.} In this case, either quadrant $II$ or $III$, whichever does not have the leak, will be able to be colored entirely by the initially colored vertices in the quadrant.

Once either quadrant $II$ or $III$ is colored, those colored vertices will be able to force until they reach the colored vertices in either quadrant $I$ or $IV$, respectively. At this point, either quadrant $I$ or $IV$ is able to force the right-hand side of the graph.

When all that is left to be colored is the quadrant with the leak, this can be done easily with the combined coloring efforts of both the originally colored vertices in that quadrant, and the colored vertices from either quadrant $I$ or $IV$.

\emph{Case 2 - The leak is in quadrant $I$ or $IV$.} In this case, the initially colored vertices in quadrants $II$ and $III$ can color the left-hand side of the graph and continue forcing until either the leak is reached or the colored vertices in either quadrant $I$ or $IV$ are reached. At this point, the quadrant without the leak can be colored entirely, and the quadrant with the leak will be forced by its two neighboring quadrants being completely colored.
\end{proof}

\begin{center}
\begin{figure}[!h]
\begin{tikzpicture}[scale=0.7]

\draw[step=1cm,black,thin] (0,0) grid (9,5);

\draw[dashed,thick] (-1,2.5)--(10,2.5);
\draw[dashed,thick] (4.5,-1)--(4.5,6);

\draw[fill=white] (0,0) circle (4pt);
\draw[fill=white] (0,1) circle (4pt);
\draw[fill=white] (0,2) circle (4pt);
\draw[fill=white] (0,3) circle (4pt);
\draw[fill=white] (0,4) circle (4pt);
\draw[fill=white] (0,5) circle (4pt);

\draw[fill=white] (1,0) circle (4pt);
\draw[fill=white] (1,1) circle (4pt);
\draw[fill=white] (1,2) circle (4pt);
\draw[fill=white] (1,3) circle (4pt);
\draw[fill=white] (1,4) circle (4pt);
\draw[fill=white] (1,5) circle (4pt);

\draw[fill=white] (2,0) circle (4pt);
\draw[fill=white] (2,1) circle (4pt);
\draw[fill=white] (2,2) circle (4pt);
\draw[fill=white] (2,3) circle (4pt);
\draw[fill=white] (2,4) circle (4pt);
\draw[fill=white] (2,5) circle (4pt);

\draw[fill=white] (3,0) circle (4pt);
\draw[fill=white] (3,1) circle (4pt);
\draw[fill=white] (3,2) circle (4pt);
\draw[fill=white] (3,3) circle (4pt);
\draw[fill=white] (3,4) circle (4pt);
\draw[fill=white] (3,5) circle (4pt);

\draw[fill=white] (4,0) circle (4pt);
\draw[fill=white] (4,1) circle (4pt);
\draw[fill=white] (4,2) circle (4pt);
\draw[fill=white] (4,3) circle (4pt);
\draw[fill=white] (4,4) circle (4pt);
\draw[fill=white] (4,5) circle (4pt);

\draw[fill=white] (5,0) circle (4pt);
\draw[fill=white] (5,1) circle (4pt);
\draw[fill=white] (5,2) circle (4pt);
\draw[fill=white] (5,3) circle (4pt);
\draw[fill=white] (5,4) circle (4pt);
\draw[fill=white] (5,5) circle (4pt);

\draw[fill=white] (6,0) circle (4pt);
\draw[fill=white] (6,1) circle (4pt);
\draw[fill=white] (6,2) circle (4pt);
\draw[fill=white] (6,3) circle (4pt);
\draw[fill=white] (6,4) circle (4pt);
\draw[fill=white] (6,5) circle (4pt);

\draw[fill=white] (7,0) circle (4pt);
\draw[fill=white] (7,1) circle (4pt);
\draw[fill=white] (7,2) circle (4pt);
\draw[fill=white] (7,3) circle (4pt);
\draw[fill=white] (7,4) circle (4pt);
\draw[fill=white] (7,5) circle (4pt);

\draw[fill=white] (8,0) circle (4pt);
\draw[fill=white] (8,1) circle (4pt);
\draw[fill=white] (8,2) circle (4pt);
\draw[fill=white] (8,3) circle (4pt);
\draw[fill=white] (8,4) circle (4pt);
\draw[fill=white] (8,5) circle (4pt);

\draw[fill=white] (9,0) circle (4pt);
\draw[fill=white] (9,1) circle (4pt);
\draw[fill=white] (9,2) circle (4pt);
\draw[fill=white] (9,3) circle (4pt);
\draw[fill=white] (9,4) circle (4pt);
\draw[fill=white] (9,5) circle (4pt);

\draw (2.25,1.25) ++(0:0cm) node [fill=white,inner sep=4pt](a){$III$};
\draw (2.25,3.75) ++(0:0cm) node [fill=white,inner sep=4pt](a){$II$};
\draw (6.75,1.25) ++(0:0cm) node [fill=white,inner sep=4pt](a){$IV$};
\draw (6.75,3.75) ++(0:0cm) node [fill=white,inner sep=4pt](a){$I$};

\end{tikzpicture}
\caption{Quadrants visualized, as discussed in Proposition \ref{array pattern prop}.}
\label{subregions array}
\end{figure}
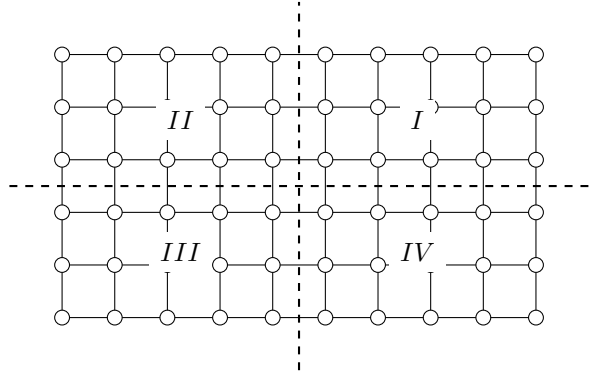
\end{center}

This also proves Theorem \ref{thm.uppergrid}. \hfill \qed

\begin{corollary}
\[Z_{(1)}(P_n \Box P_n) = n. \]
\end{corollary}

\begin{proof}
By Proposition \ref{boxprodleaks}, $Z_{(1)}(P_n \Box P_n) \le 2n-n \le n$. A lower bound is achieved by Propositions \ref{zf_chain} and \ref{zf_array prop}: $Z_{(1)}(P_n \Box P_n) \ge Z_{(0)}(P_n \Box P_n) = n.$
\end{proof}

It is interesting to note that for the square grid (for $n > 2$), $n$ vertices on any one side of the grid is sufficient for zero forcing, but not sufficient for 1-forcing.

\begin{definition} \label{improved bar grid}
(Initial pattern for Theorem \ref{thm.bargrid}). 

In Theorem \ref{thm.bargrid}, we introduce an improved bound on the 1-forcing number of grid graphs for certain values of $m$ and $n$. The initially colored vertices which give us this improved bound are given by the coordinates (using the coordinate system as in Definition \ref{coordinate def}):

For $m$ even:

$$\left(1, \frac{m}{2}\right) \text{ and } \left(1, \frac{m}{2}+1\right)$$

as well as

$$ (2,2), (2,3), (2,4), \dots , (2,m-1). $$

For $m$ odd:

$$ \left(1, \Bigl \lceil \frac{m}{2} \Bigr \rceil \right), \text{ and either } \left(1, \Bigl \lceil \frac{m}{2} \Bigr \rceil +1 \right) \text{ or } \left(1, \Bigl \lceil \frac{m}{2} \Bigr \rceil -1 \right)$$

as well as

$$ (2,2), (2,3), (2,4), \dots , (2,m-1). $$

There will be $m$ initially colored vertices after this coloring pattern has been completed. An example of this pattern is pictured in Figure \ref{bar grid}.

\end{definition}

\begin{theorem} \label{thm.bargrid}
For the Box product of paths, $P_n \Box P_m$, with $\bigl\lfloor \frac{m}{2} \bigr\rfloor + 2 \ge n$ \[Z_{(1)}(P_n \Box P_m) \leq m. \]
\end{theorem}

\begin{proof}

Just as in Proposition \ref{array pattern prop}, we will first visualize the $P_{n} \Box P_{m}$ grid as split into four regions. However, this will not be an even split as before; an example of these new regions is given in Figure \ref{bar grid figure}. A vertical cut is still pictured halfway down the side of length $m$, and the new horizontal cut is pictured after the second vertex from the top down the side of length $n$. We will then have the following cases, using the initially colored vertices as found in Definition \ref{improved bar grid}:

\emph{Case 1 - The leak is in region $I$ or $II$.} In this case, either region $I$ or $II$, whichever does not have the leak, will be able to be colored entirely by the initially colored vertices in the region.

Once either region $I$ or $II$ is colored, those colored vertices will be able to force until they reach the ``bottom corner," the vertex in the lower right- or left-hand side in either region $IV$ or $III$, respectively. At this point, the entire graph can be colored by Lemma \ref{lem.sides}.

\emph{Case 2 - The leak is in region $III$ or $IV$.} In this case, the initially colored vertices in regions $I$ and $II$ can be entirely colored, and they can begin forcing into regions $IV$ and $III$, respectively, until the leak is reached. The condition $\bigl\lfloor \frac{m}{2} \bigr\rfloor + 2 \ge n$ guarantees that regardless of the location of the leak, one of the ``bottom corners'' will be forced, at which point two sides are colored and Lemma \ref{lem.sides} guarantees the entire graph can be colored.

\end{proof}

This proof is limited to those grids with $\bigl\lfloor \frac{m}{2} \bigr\rfloor + 2 \ge n$ because if $n$ were to be any larger, it could not be guaranteed that a ``bottom corner" in either region $III$ or $IV$ could be reached by the initially colored vertices. We have found that this condition is necessary for this pattern in order to ensure the entire graph may be colored.

\begin{center}
\begin{figure}[!h]
\begin{tikzpicture}[scale=0.7]

\draw[step=1cm,black,thin] (0,0) grid (9,5);

\draw[dashed,thick] (-1,3.5)--(10,3.5);
\draw[dashed,thick] (4.5,-1)--(4.5,6);

\draw[fill=white] (0,0) circle (4pt);
\draw[fill=white] (0,1) circle (4pt);
\draw[fill=white] (0,2) circle (4pt);
\draw[fill=white] (0,3) circle (4pt);
\draw[fill=white] (0,4) circle (4pt);
\draw[fill=white] (0,5) circle (4pt);

\draw[fill=white] (1,0) circle (4pt);
\draw[fill=white] (1,1) circle (4pt);
\draw[fill=white] (1,2) circle (4pt);
\draw[fill=white] (1,3) circle (4pt);
\draw[fill=white] (1,4) circle (4pt);
\draw[fill=white] (1,5) circle (4pt);

\draw[fill=white] (2,0) circle (4pt);
\draw[fill=white] (2,1) circle (4pt);
\draw[fill=white] (2,2) circle (4pt);
\draw[fill=white] (2,3) circle (4pt);
\draw[fill=white] (2,4) circle (4pt);
\draw[fill=white] (2,5) circle (4pt);

\draw[fill=white] (3,0) circle (4pt);
\draw[fill=white] (3,1) circle (4pt);
\draw[fill=white] (3,2) circle (4pt);
\draw[fill=white] (3,3) circle (4pt);
\draw[fill=white] (3,4) circle (4pt);
\draw[fill=white] (3,5) circle (4pt);

\draw[fill=white] (4,0) circle (4pt);
\draw[fill=white] (4,1) circle (4pt);
\draw[fill=white] (4,2) circle (4pt);
\draw[fill=white] (4,3) circle (4pt);
\draw[fill=white] (4,4) circle (4pt);
\draw[fill=white] (4,5) circle (4pt);

\draw[fill=white] (5,0) circle (4pt);
\draw[fill=white] (5,1) circle (4pt);
\draw[fill=white] (5,2) circle (4pt);
\draw[fill=white] (5,3) circle (4pt);
\draw[fill=white] (5,4) circle (4pt);
\draw[fill=white] (5,5) circle (4pt);

\draw[fill=white] (6,0) circle (4pt);
\draw[fill=white] (6,1) circle (4pt);
\draw[fill=white] (6,2) circle (4pt);
\draw[fill=white] (6,3) circle (4pt);
\draw[fill=white] (6,4) circle (4pt);
\draw[fill=white] (6,5) circle (4pt);

\draw[fill=white] (7,0) circle (4pt);
\draw[fill=white] (7,1) circle (4pt);
\draw[fill=white] (7,2) circle (4pt);
\draw[fill=white] (7,3) circle (4pt);
\draw[fill=white] (7,4) circle (4pt);
\draw[fill=white] (7,5) circle (4pt);

\draw[fill=white] (8,0) circle (4pt);
\draw[fill=white] (8,1) circle (4pt);
\draw[fill=white] (8,2) circle (4pt);
\draw[fill=white] (8,3) circle (4pt);
\draw[fill=white] (8,4) circle (4pt);
\draw[fill=white] (8,5) circle (4pt);

\draw[fill=white] (9,0) circle (4pt);
\draw[fill=white] (9,1) circle (4pt);
\draw[fill=white] (9,2) circle (4pt);
\draw[fill=white] (9,3) circle (4pt);
\draw[fill=white] (9,4) circle (4pt);
\draw[fill=white] (9,5) circle (4pt);

\draw (0.5,2.5) ++(0:0cm) node [fill=white,inner sep=2pt](a){$III$};
\draw (0.5,4.5) ++(0:0cm) node [fill=white,inner sep=4pt](a){$II$};
\draw (5.5,4.5) ++(0:0cm) node [fill=white,inner sep=4pt](a){$I$};
\draw (5.5,2.5) ++(0:0cm) node [fill=white,inner sep=3pt](a){$IV$};

\end{tikzpicture}
\caption{Visualization of the regions used in the proof of Theorem \ref{thm.bargrid}.}
\label{bar grid figure}
\end{figure}
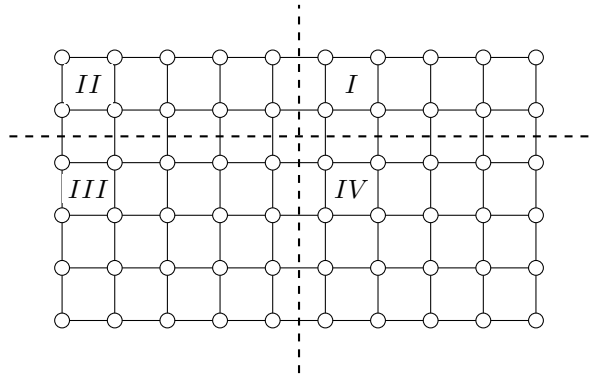
\end{center}

\begin{center}
\begin{figure}[h]
\centering
\begin{tikzpicture}[scale=0.7]

\draw[step=1cm,black,thin] (0,0) grid (9,6);

\draw [decorate,decoration={brace,amplitude=10pt},xshift=0pt,yshift=0pt] (0,7.5) -- (9,7.5) node [black,midway,xshift=0pt,yshift=15pt] {\footnotesize $m$};

\draw [decorate,decoration={brace,amplitude=10pt},xshift=0pt,yshift=0pt] (0,6.5) -- (4.5,6.5) node [black,midway,xshift=0pt,yshift=15pt] {\footnotesize $\sfrac{m}{2}$};

\draw[fill=white] (3,6) circle (4pt);
\draw[fill] (3,5) circle (4pt);
\draw[fill=white] (3,4) circle (4pt);
\draw[fill=white] (3,3) circle (4pt);
\draw[fill=white] (2,4) circle (4pt);
\draw[fill=white] (2,6) circle (4pt);
\draw[fill] (4,5) circle (4pt);

\draw[fill=white] (0,1) circle (4pt);
\draw[fill=white] (0,2) circle (4pt);
\draw[fill=white] (0,3) circle (4pt);
\draw[fill=white] (0,4) circle (4pt);
\draw[fill=white] (0,5) circle (4pt);
\draw[fill=white] (0,6) circle (4pt);

\draw[fill=white] (1,1) circle (4pt);
\draw[fill=white] (1,2) circle (4pt);
\draw[fill=white] (1,3) circle (4pt);
\draw[fill=white] (1,4) circle (4pt);
\draw[fill] (1,5) circle (4pt);
\draw[fill=white] (1,6) circle (4pt);

\draw[fill=white] (5,1) circle (4pt);
\draw[fill=white] (5,2) circle (4pt);
\draw[fill=white] (5,3) circle (4pt);
\draw[fill=white] (5,4) circle (4pt);
\draw[fill] (5,5) circle (4pt);
\draw[fill] (5,6) circle (4pt);

\draw[fill=white] (6,1) circle (4pt);
\draw[fill=white] (6,2) circle (4pt);
\draw[fill=white] (6,3) circle (4pt);
\draw[fill=white] (6,4) circle (4pt);
\draw[fill] (6,5) circle (4pt);
\draw[fill=white] (6,6) circle (4pt);

\draw[fill=white] (7,1) circle (4pt);
\draw[fill=white] (7,2) circle (4pt);
\draw[fill=white] (7,3) circle (4pt);
\draw[fill=white] (7,4) circle (4pt);
\draw[fill] (7,5) circle (4pt);
\draw[fill=white] (7,6) circle (4pt);

\draw[fill=white] (8,1) circle (4pt);
\draw[fill=white] (8,2) circle (4pt);
\draw[fill=white] (8,3) circle (4pt);
\draw[fill=white] (8,4) circle (4pt);
\draw[fill] (8,5) circle (4pt);
\draw[fill=white] (8,6) circle (4pt);

\draw[fill=white] (9,1) circle (4pt);
\draw[fill=white] (9,2) circle (4pt);
\draw[fill=white] (9,3) circle (4pt);
\draw[fill=white] (9,4) circle (4pt);
\draw[fill=white] (9,5) circle (4pt);
\draw[fill=white] (9,6) circle (4pt);

\draw[fill=white] (0,0) circle (4pt);
\draw[fill=white] (1,0) circle (4pt);
\draw[fill=white] (2,0) circle (4pt);
\draw[fill=white] (3,0) circle (4pt);
\draw[fill=white] (4,0) circle (4pt);
\draw[fill=white] (5,0) circle (4pt);
\draw[fill=white] (6,0) circle (4pt);
\draw[fill=white] (7,0) circle (4pt);
\draw[fill=white] (8,0) circle (4pt);
\draw[fill=white] (9,0) circle (4pt);

\draw[fill] (2,5) circle (4pt);
\draw[fill=white] (2,3) circle (4pt);
\draw[fill=white] (2,2) circle (4pt);
\draw[fill=white] (2,1) circle (4pt);

\draw[fill=white] (3,2) circle (4pt);
\draw[fill=white] (3,1) circle (4pt);

\draw[fill] (4,6) circle (4pt);
\draw[fill=white] (4,4) circle (4pt);
\draw[fill=white] (4,3) circle (4pt);
\draw[fill=white] (4,2) circle (4pt);
\draw[fill=white] (4,1) circle (4pt);

\draw [decorate,decoration={brace,amplitude=10pt},xshift=0pt,yshift=0pt] (-0.6,0) -- (-0.6, 6) node [black,midway,xshift=-17pt,yshift=0pt] {\footnotesize $n$};

\end{tikzpicture}
\caption{An example of the pattern in Definition \ref{improved bar grid} for Theorem \ref{thm.bargrid}, shown here on a $P_{7} \Box P_{10}$ graph.}
\label{bar grid}
\end{figure}
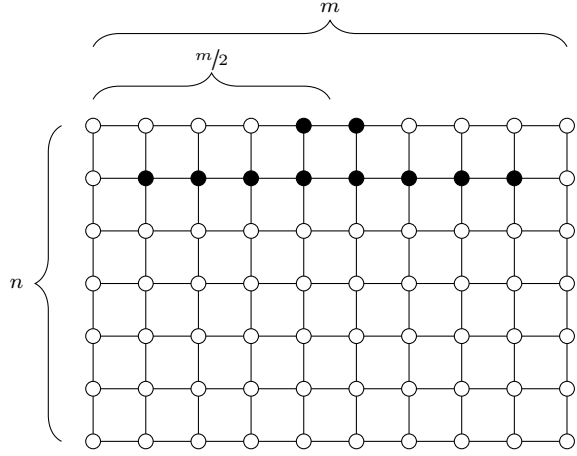
\end{center}

\begin{definition}\label{def.wing} (Initial pattern for Theorem \ref{thm.mn5}).

Let $L, R, H$ be determined by $m$ and $n$ as in Table \ref{even pattern values table} (if $m$ is even) or \ref{odd pattern values table} (if $m$ is odd). The set of initially colored vertices  includes $(i, 2)$ for $i = 1, \ldots, L$ (if $L=0$, color no vertices in this way), $(m-i, 2)$ for $i = 1, \ldots, R$; and both $\left(\left\lfloor \frac m 2 \right\rfloor, i\right)$ and $\left(\left\lfloor \frac m 2 \right\rfloor +1, i\right)$ for $i=1, \ldots , H$ (if $H=0$, color no vertices in this way). An illustration can be seen in Figure \ref{new array figure}.
\label{new array def}
\end{definition}

\begin{theorem}\label{thm.mn5}
For the Box product of paths, $P_n \Box P_m$, with $n \le m \le n+5$ and $m \ge 7$, \[Z_{(1)}(P_n \Box P_m) \leq m. \]

\end{theorem}

\begin{proof}
Let $L, R, H$ be as in Tables \ref{even pattern values table} and \ref{odd pattern values table}.

Let $S$ be the initially colored set as described in Definition \ref{new array def}.

Just as before, we will first visualize the $P_{n} \Box P_{m}$ grid as split into four regions using two orthogonal cuts. Region $I$ is the set of vertices, $(x,y)$ for which $x \ge \left\lfloor \frac m 2 \right\rfloor +1$ and $y \le H$. Region $II$ is the set of vertices, $(x,y)$ for which $x \le \left\lfloor \frac m 2 \right\rfloor$ and $y \le H$. Region $III$ is the set of vertices, $(x,y)$ for which $x \le \left\lfloor \frac m 2 \right\rfloor$ and $y \ge H+1$. Finally, region $IV$ is the set of vertices, $(x,y)$ for which $x \ge \left\lfloor \frac m 2 \right\rfloor +1$ and $y \ge H+1$.

\emph{Case 1 - The leak is in region $II$.} In this case region $I$ will be able to be colored entirely as follows. The vertices $\left( \lfloor \frac{m}{2} \rfloor +1, i \right)$ will force $\left( \lfloor \frac{m}{2} \rfloor +2, i \right)$ for $i = 1, \ldots, H-1$. Inductively, the vertices $\left( \lfloor \frac{m}{2} \rfloor +k+1, i \right)$ will force $\left( \lfloor \frac{m}{2} \rfloor+k+2, i \right)$ for $i=i, \ldots, H-1-k$ and for $k=1, \ldots, H-2$. When $k=H-2$, we have that $\left( \lfloor \frac{m}{2} \rfloor + H-1, 1 \right)$ will force $\left( \lfloor \frac{m}{2} \rfloor+H, 1 \right)$. A simple inspection of Tables \ref{even pattern values table} and \ref{odd pattern values table} demonstrate that $(m-R-1,1)$ becomes colored. From there, $(m-R-1,1)$ can force $(m-R-2,1)$, and so on until $(m,1)$ is colored. After this, the vertices $\left( \lfloor \frac{m}{2} \rfloor +1, 1 \right), \ldots ,\left(m, 1 \right)$ can iteratively force all of region $I$.

Once region $I$ is colored, the vertices $\left(\lfloor \frac{m}{2} \rfloor +1 , H \right), \left(\lfloor \frac{m}{2} \rfloor +2, H \right), \ldots, \left(m, H \right)$ can force $\left(\lfloor \frac{m}{2} \rfloor +1, H+1 \right), \left(\lfloor \frac{m}{2} \rfloor +1, H+1 \right), \ldots, \left(m, H+1 \right)$. Inductively, the vertices \[ \left(\left\lfloor\frac{m}{2} \right\rfloor +i, H+i \right), \left(\left\lfloor \frac{m}{2} \right\rfloor +1+i, H+i \right), \ldots, \left(m, H+i \right)\] can force \[\left(\left\lfloor \frac{m}{2} \right\rfloor +i+1, H+i+1 \right), \left(\left\lfloor \frac{m}{2} \right\rfloor +1+i+1, H+i+1 \right), \ldots, \left(m, H+i+1 \right).\] Note that there may be a leak on $\left( \lfloor \frac{m}{2} \rfloor, H \right)$; therefore, each of these steps guarantees once less force than the previous step (resulting in, at least, a triangle pattern). Therefore, this process applies potentially $\lceil \frac{m}{2} \rceil$ times.

We now show that $(m,n)$ (the bottom-rightmost vertex) becomes colored from these $\lceil \frac{m}{2} \rceil$ iterations; however, this depends on $m$ and $n$. When $i = \lceil \frac{m}{2} \rceil$, the bottom rightmost vertex that is colored is (or would be) $\left(m, H+i+1 \right) = \left(m, H+\lceil \frac{m}{2} \rceil+1 \right)$. Based on Tables \ref{even pattern values table} and \ref{odd pattern values table}, this is given by the following:
\begin{itemize}
\item For $m$ even and $m-n = 1$, $H+\left\lceil \frac{m}{2} \right\rceil+1 = m = n+1$;
\item For $m$ even and $m-n = 2$, $H+\left\lceil \frac{m}{2} \right\rceil+1 = m = n+2$;
\item For $m$ even and $m-n = 3$, $H+\left\lceil \frac{m}{2} \right\rceil+1 = m-1 = n+2$;
\item For $m$ even and $m-n = 4$, $H+\left\lceil \frac{m}{2} \right\rceil+1 = m-1 = n+3$;
\item For $m$ even and $m-n = 5$, $H+\left\lceil \frac{m}{2} \right\rceil+1 = m-2 = n+3$;
\item For $m$ odd and $m-n = 1$, $H+\left\lceil \frac{m}{2} \right\rceil+1 = m+1 = n+2$;
\item For $m$ odd and $m-n = 2$, $H+\left\lceil \frac{m}{2} \right\rceil+1 = m+1 = n+1$;
\item For $m$ odd and $m-n = 3$, $H+\left\lceil \frac{m}{2} \right\rceil+1 = m = n+3$;
\item For $m$ odd and $m-n = 4$, $H+\left\lceil \frac{m}{2} \right\rceil+1 = m-1 = n+3$;
\item For $m$ odd and $m-n = 5$, $H+\left\lceil \frac{m}{2} \right\rceil+1 = m-2 = n+3$
\end{itemize}
While the vertex $(m,n+k)$ is not a valid vertex for any $k > 0$, the list above demonstrates that, if the grid were extended, the pattern {\it could} continue to force until the vertex $(m,n+k)$, is colored and in particular, the vertex $(m,n)$ and vertices of the form $(m,j)$ (for $j=1, \ldots, n$) will be colored. Since the leak is in region I, this side is sufficient to force all vertices in region IV as well as all vertices of the form $\left(\left\lfloor \frac m 2 \right\rfloor, j\right)$ for $j = 1, \ldots, n$.

Since the leak is in region II, the vertices $\left(\left\lfloor \frac m 2 \right\rfloor, k\right)$ will force $\left(\left\lfloor \frac{m}{2} \right\rfloor -1, k\right)$ for $k = H+1, \ldots, n$, and inductively $\left(\left\lfloor \frac{m}{2} \right\rfloor - i, k\right)$ will force $\left(\left\lfloor \frac{m}{2} \right\rfloor - i -1, k\right)$ for $k = H+1+i, \ldots, n$ for potentially $i = 1, \ldots, n-H$. As before, since there may be a leak, each of these steps only guarantees once less force than the previous step, resulting the in the same triangle pattern. Similar to before, we determine the bottom-leftmost colored (or potentially colored vertex, if the grid were extended) vertex  step $i = n-H$. Indeed, at step $i = n-H$ and for $k=n$, $\left(\left\lfloor \frac{m}{2} \right\rfloor - i -1, n\right)= \left(\left\lfloor \frac m 2 \right\rfloor+H - n -1, n\right)$

\begin{itemize}
\item For $m$ even and $m-n = 1$, $\left\lfloor \frac{m}{2} \right\rfloor+H - n -1 = -1$;
\item For $m$ even and $m-n = 2$, $\left\lfloor \frac{m}{2} \right\rfloor+H - n -1 = 0$;
\item For $m$ even and $m-n = 3$, $\left\lfloor \frac{m}{2} \right\rfloor+H - n -1= 0$;
\item For $m$ even and $m-n = 4$, $\left\lfloor \frac{m}{2} \right\rfloor+H - n -1 = 1$;
\item For $m$ even and $m-n = 5$, $\left\lfloor \frac{m}{2} \right\rfloor+H - n -1 = 1$;
\item For $m$ odd and $m-n = 1$, $\left\lfloor \frac{m}{2} \right\rfloor+H - n -1 = -1$;
\item For $m$ odd and $m-n = 2$, $\left\lfloor \frac{m}{2} \right\rfloor+H - n -1 = 0$;
\item For $m$ odd and $m-n = 3$, $\left\lfloor \frac{m}{2} \right\rfloor+H - n -1 =0$;
\item For $m$ odd and $m-n = 4$, $\left\lfloor \frac{m}{2} \right\rfloor+H - n -1 = 0$;
\item For $m$ odd and $m-n = 5$, $\left\lfloor \frac{m}{2} \right\rfloor+H - n -1 = 0$
\end{itemize}
Therefore, $(1,n)$ is colored; hence we have colored two sides of the grid. By Lemma \ref{lem.sides}, this is sufficient to color the entire graph.

\emph{Case 2 - The leak is in region $I$}. By symmetry, the proof in Case 1 holds for even $m$.  For odd $m$, the proof is very similar and is omitted for brevity.

\emph{Case 3 - The leak is in region $III$ or $IV$}. In this case, regions $I$ and $II$ will be colored by the initial steps in Cases 1 and 2. If the leak is in region $III$, then we can color region $IV$ just as in Case 1, after which, we have colored two consecutive sides, and we are done by Lemma \ref{lem.sides}. If the leak is in region $IV$, then we can color region $III$ similarly, and we are done by Lemma \ref{lem.sides}.

\end{proof}

\begin{center}
\begin{figure}[h]
\centering
\begin{tikzpicture}[scale=0.7]

\draw[step=1cm,black,thin] (0,0) grid (1,6);

\draw[step=1cm,black,thin] (2,0) grid (7,6);

\draw[step=1cm,black,thin] (8,0) grid (9,6);

\draw[dashed,thin] (1,0)--(2,0);
\draw[dashed,thin] (1,1)--(2,1);
\draw[dashed,thin] (1,2)--(2,2);
\draw[dashed,thin] (1,3)--(2,3);
\draw[dashed,thin] (1,4)--(2,4);
\draw[dashed,thin] (1,5)--(2,5);
\draw[dashed,thin] (1,6)--(2,6);

\draw[dashed,thin] (7,0)--(8,0);
\draw[dashed,thin] (7,1)--(8,1);
\draw[dashed,thin] (7,2)--(8,2);
\draw[dashed,thin] (7,3)--(8,3);
\draw[dashed,thin] (7,4)--(8,4);
\draw[dashed,thin] (7,5)--(8,5);
\draw[dashed,thin] (7,6)--(8,6);

\draw[dashed,thin,gray] (4.5,8.5)--(4.5,6.2);
\draw[dashed,thin,gray] (5,8.5)--(5,6.2);

\draw [decorate,decoration={brace,amplitude=10pt},xshift=0pt,yshift=0pt] (0,8.5) -- (9,8.5) node [black,midway,xshift=0pt,yshift=15pt] {\footnotesize $m$};

\draw (2.6,8.2) ++(0:0cm) node [fill=white,inner sep=3pt](a){$\frac{m}{2}$, for even $m$};

\draw (7,7.9) ++(0:0cm) node [fill=white,inner sep=3pt](a){$\lceil \sfrac{m}{2} \rceil$ for odd $m$};

\draw [decorate,decoration={brace,amplitude=10pt},xshift=0pt,yshift=0pt] (0.6,6.5) -- (2.4,6.5) node [black,midway,xshift=0pt,yshift=15pt] {\footnotesize $L$};

\draw [decorate,decoration={brace,amplitude=10pt},xshift=0pt,yshift=0pt] (6.6,6.5) -- (8.4,6.5) node [black,midway,xshift=0pt,yshift=15pt] {\footnotesize $R$};

\draw[fill=white] (3,6) circle (4pt);
\draw[fill=white] (3,5) circle (4pt);
\draw[fill=white] (3,4) circle (4pt);
\draw[fill=white] (3,3) circle (4pt);
\draw[fill=white] (2,4) circle (4pt);
\draw[fill=white] (2,6) circle (4pt);
\draw[fill] (4,5) circle (4pt);

\draw[fill=white] (0,1) circle (4pt);
\draw[fill=white] (0,2) circle (4pt);
\draw[fill=white] (0,3) circle (4pt);
\draw[fill=white] (0,4) circle (4pt);
\draw[fill=white] (0,5) circle (4pt);
\draw[fill=white] (0,6) circle (4pt);

\draw[fill=white] (1,1) circle (4pt);
\draw[fill=white] (1,2) circle (4pt);
\draw[fill=white] (1,3) circle (4pt);
\draw[fill=white] (1,4) circle (4pt);
\draw[fill] (1,5) circle (4pt);
\draw[fill=white] (1,6) circle (4pt);

\draw[fill=white] (5,1) circle (4pt);
\draw[fill=white] (5,2) circle (4pt);
\draw[fill] (5,3) circle (4pt);
\draw[fill] (5,4) circle (4pt);
\draw[fill] (5,5) circle (4pt);
\draw[fill] (5,6) circle (4pt);

\draw[fill=white] (6,1) circle (4pt);
\draw[fill=white] (6,2) circle (4pt);
\draw[fill=white] (6,3) circle (4pt);
\draw[fill=white] (6,4) circle (4pt);
\draw[fill=white] (6,5) circle (4pt);
\draw[fill=white] (6,6) circle (4pt);

\draw[fill=white] (7,1) circle (4pt);
\draw[fill=white] (7,2) circle (4pt);
\draw[fill=white] (7,3) circle (4pt);
\draw[fill=white] (7,4) circle (4pt);
\draw[fill] (7,5) circle (4pt);
\draw[fill=white] (7,6) circle (4pt);

\draw[fill=white] (8,1) circle (4pt);
\draw[fill=white] (8,2) circle (4pt);
\draw[fill=white] (8,3) circle (4pt);
\draw[fill=white] (8,4) circle (4pt);
\draw[fill] (8,5) circle (4pt);
\draw[fill=white] (8,6) circle (4pt);

\draw[fill=white] (9,1) circle (4pt);
\draw[fill=white] (9,2) circle (4pt);
\draw[fill=white] (9,3) circle (4pt);
\draw[fill=white] (9,4) circle (4pt);
\draw[fill=white] (9,5) circle (4pt);
\draw[fill=white] (9,6) circle (4pt);

\draw[fill=white] (0,0) circle (4pt);
\draw[fill=white] (1,0) circle (4pt);
\draw[fill=white] (2,0) circle (4pt);
\draw[fill=white] (3,0) circle (4pt);
\draw[fill=white] (4,0) circle (4pt);
\draw[fill=white] (5,0) circle (4pt);
\draw[fill=white] (6,0) circle (4pt);
\draw[fill=white] (7,0) circle (4pt);
\draw[fill=white] (8,0) circle (4pt);
\draw[fill=white] (9,0) circle (4pt);

\draw[fill] (2,5) circle (4pt);
\draw[fill=white] (2,3) circle (4pt);
\draw[fill=white] (2,2) circle (4pt);
\draw[fill=white] (2,1) circle (4pt);

\draw[fill=white] (3,2) circle (4pt);
\draw[fill=white] (3,1) circle (4pt);

\draw[fill] (4,6) circle (4pt);
\draw[fill] (4,4) circle (4pt);
\draw[fill] (4,3) circle (4pt);
\draw[fill=white] (4,2) circle (4pt);
\draw[fill=white] (4,1) circle (4pt);

\draw [decorate,decoration={brace,amplitude=10pt},xshift=0pt,yshift=0pt] (-1.6,0) -- (-1.6, 6) node [black,midway,xshift=-17pt,yshift=0pt] {\footnotesize $n$};

\draw [decorate,decoration={brace,amplitude=10pt},xshift=0pt,yshift=0pt] (-0.6,3) -- (-0.6, 6) node [black,midway,xshift=-17pt,yshift=0pt] {\footnotesize $H$};

\end{tikzpicture}
\caption{An example of the pattern for Theorem \ref{thm.mn5}, with $L$, $R$, and $H$ values found in Tables \ref{even pattern values table} and \ref{odd pattern values table}.}
\label{new array figure}
\end{figure}
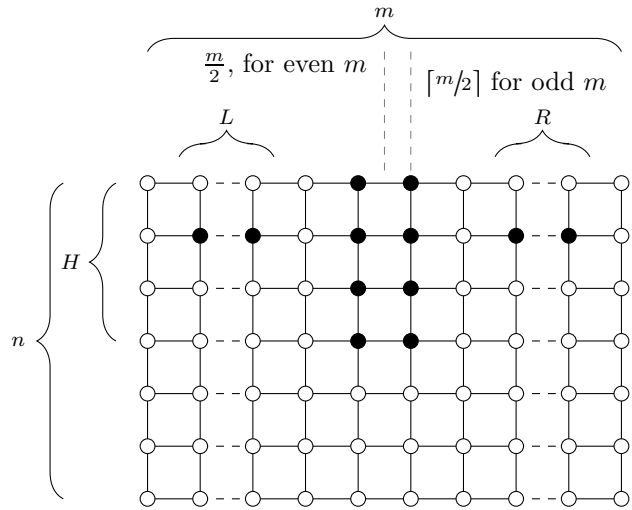
\end{center}

\def\arraystretch{1.5}
\begin{center}
\begin{table} [h]
    \begin{tabular}{| c | c | c | c |}
    \hline
    $m-n$ & $L$ & $R$ & $H$
    \\ \hline \hline
    1 or 2 & 1 & 1 & $\frac{m-2}{2}$
    \\ \hline
    3 or 4 & 2 & 2 & $\frac{m-4}{2}$
    \\ \hline
    5 & 3 & 3 & $\frac{m-6}{2}$
    \\ \hline
    \end{tabular}
\caption{Values for $L$, $R$, and $H$ given an even $m$ value.}
\label{even pattern values table}
\end{table}

\end{center}

\def\arraystretch{1.5}
\begin{center}
\begin{table} [h]
    \begin{tabular}{| c | c | c | c |}
    \hline
    $m-n$ & $L$ & $R$ & $H$
    \\ \hline \hline
    1 or 2 & 0 & 1 & $\frac{m-1}{2}$
    \\ \hline
    3 & 1 & 2 & $\frac{m-3}{2}$
    \\ \hline
    4 & 2 & 3 & $\frac{m-5}{2}$
    \\ \hline
    5 & 3 & 4 & $\frac{m-7}{2}$
    \\ \hline
    \end{tabular}
\caption{Values for $L$, $R$, and $H$ given an odd $m$ value.}
\label{odd pattern values table} 
\end{table}
\end{center}

\def\arraystretch{1.5}
\begin{center}
\begin{table} [h]
    \begin{tabular}{ c  c || c | c | c | c | c | c|}
    \cline{3-8}
     & & \multicolumn{6}{|c|}{$n$} \\
    \cline{3-8}
      &  & $2$ & $3$ & $4$ & $5$ & $6$ & $7$ \\ \hline \hline

    \multicolumn{1}{|c|}{} & $2$ & $2$ & $3$ & $4$ & $4$ & 4 & $\hdots$ \\ \cline{2-8}

    \multicolumn{1}{|c|}{ \multirow{5}{1em}{$m$} } & $3$ & $3$ & $3$ & $4$ & $5$ & $6$ & $\hdots$ \\
    \cline{2-8}
    \multicolumn{1}{|c|}{} & $4$ & $4$ & $4$ & $4$ & $5$& $6$ & $7$ \\
    \cline{2-8}
    \multicolumn{1}{|c|}{} & $5$ & $4$ & $5$ & $5$ & $5$ & $6$ & 7\\
    \cline{2-8}
    \multicolumn{1}{|c|}{} & $6$ & $4$ & $6$ & $6$ & $6$ & $6$ & 7\\
    \cline{2-8}
    \multicolumn{1}{|c|}{} & $7$ & $4$ & $6$ & $7$ & $7$ & $7$ & 7 \\
    \cline{2-8}
    \multicolumn{1}{|c|}{} & $8$ & $\vdots$ & $\vdots$ & $8$ & $\leq 8$ & $\leq 8$ & $\le 8$ \\
   \hline
    \end{tabular}
\caption{Summary of values and bounds of $Z_{(1)}(P_n \Box P_m)$.} \label{tab.grids}
\end{table}
\end{center}

For grid graphs, we remark that Table \ref{tab.grids} indicates for select values of $n \le m$ that $Z_{(1)}(P_n \Box P_m) = \min (2n, m)$. While we have proven the upper bound  $Z_{(1)}(P_n \Box P_m) \le \min (2n, m)$  for infinite values of $n$ and $m$, a justification as to why this is optimal is elusive.

\section{Integer Program Algorithm For Leaky Forcing} \label{sec.alg}

We implement a constraint generation method to calculate the $\ell$-forcing number similar to \cite{compzf}, using Proposition \ref{prop.fort}. 




In the course of solving the problem of finding an $\ell$-forcing set, this integer program is run first with a default fort containing every vertex in the graph if no other forts are initialized. Once a solution vector is found, it is tested as an $\ell$-forcing set. If it is successful, then we have completed our task. If not, the set of vertices that the solution set was unable to color becomes a fort, which is then added to the set of forts $F$ used in this integer program. This process continues until the solution vector is a successful $\ell$-forcing set.

Suppose in an application of $\ell$-forcing, there is a risk of a vertex in the $\ell$-forcing set becoming uncolored. If this risk exists, the designer of the system may wish to create some redundancy in the system to be more certain that the entire graph may still become colored. If redundancy is advantageous, then we would only need to add the requirement that two vertices from each fort are initially colored. Then, if any one colored vertex is ``turned off" for any reason, every fort is still guaranteed to have at least one colored vertex initially, so the colored vertices still form an $\ell$-forcing set. The integer program's constraint can be easily updated to solve this problem with redundant vertices in the $\ell$-forcing set.

The sets used in this integer program are a set $V$ of vertices of the graph and a set $F$ containing the forts of the graph. The information for each fort is stored in the binary parameter $f_{k,v}$ in which a value of 1 indicates that vertex $v$ is in fort $k$, and a value of 0 indicates that vertex $v$ is not in fort $k$.

The decision variable of the integer program is a $\lvert V \rvert \times 1$  binary vector. Each element $x_v$ corresponds to a vertex $v$ in the graph. An $x_v$ value of 1 indicates that $v$ must be in the proposed $\ell$-forcing set and a value of 0 indicates that $v$ is not in the proposed $\ell$-forcing set.

The objective function ensures that the proposed $\ell$-forcing set contains as few vertices as possible. A set containing all but one vertex of a graph would be a successful $\ell$-forcing set, but we are interested in the smallest possible size for such a set.

Constraint (\ref{zeroforcingconstraint}) is a set covering constraint which requires each fort to have at least one vertex in the solution vector. As stated in \cite{compzf}, this is a sufficient requirement to ensure that the optimum result is equal to the $\ell$-forcing number for the given graph, once the solution vector is in fact a successful $\ell$-forcing set. \\

Sets

\begin{tabular}{ll}
$V$ & set of vertices \\
$F$ & set of forts \\
\end{tabular}
\\

Parameters

\begin{tabular}{lll}
$f_{k, v}$ & $k\in F, v \in V$ & membership of vertex $v$ in fort $k$ \\
\end{tabular}
\\

Variables

\begin{tabular}{lll}
$x_{v}$ & $v \in V$ & use of vertex $v$ in the $\ell$-forcing set
\end{tabular}
\setcounter{equation}{0}
\\

Objective

\begin{align}
\text{minimize} \quad
& \sum_{v \in V} x_{v}
\end{align}

Constraints

\begin{align} \label{zeroforcingconstraint}
\sum\limits_{v \in V }^{} f_{k,v} x_v \geq 1 && \forall k \in F \\
x_v \in \{ 0,1 \} && \forall v \in V
\end{align}

In Algorithm \ref{zeroforcing}, we provide the process used in creating an $\ell$-forcing set for a graph $G$ with $\ell$ leaks. $\ell$-forcing on a graph with no leaks still follows this algorithm with a value of $\ell = 0$. The inputs are the graph $G$ (and set $V$, the vertex set of $G$), a set $R \subset V$ containing vertices which must be included in the $\ell$-forcing set (if any such vertices exist), and the number of leaks $\ell$. The set $R$ is used if there are vertices that we want to be in the $\ell$-forcing set, either for application-specific reasons or in an attempt to get an output more quickly. The generation of minimal forts in lines 2 and 4 ensures that each of those vertices will be in the output, since each fort must be covered and each of those forts will each only contain one vertex. The output of this algorithm is an optimal $\ell$-forcing set $Z \subset V$ for $G$ with $\ell$ leaks.

\begin{algorithm}
\caption{Finding an $\ell$-Forcing Set for a Graph $G$ with $\ell$ leaks \newline \textbf{Input:} A graph $G$ with corresponding vertex set $V$, number of leaks $\ell$, set of vertices $R$ required to be in the $\ell$-forcing set \newline \textbf{Output:} Optimal $\ell$-forcing set, $Z$} \label{zeroforcing}

\begin{algorithmic}[1]

\For{A vertex $r$,$\forall r \in V(G)$ with degree less than or equal to $\ell$}
\State{Generate a minimal fort containing the vertex $r$, add the fort to $F$}
\EndFor

\For{A vertex $r$, $\forall r \in R$ }
\State{Generate a minimal fort containing the vertex $r$, add the fort to $F$}
\EndFor

\State{Run IP in Section \ref{sec.alg} with the set of forts $F$ to generate $Z$}

\While{$Z$ is not an $\ell$-forcing set on $G$}
\State{Generate a minimal fort containing the vertices left uncolored by $Z$, add the fort to $F$}
\State{Run IP in Section \ref{sec.alg} with the updated set of forts $F$ to generate $Z$}
\EndWhile

\end{algorithmic}
\end{algorithm}


We first tested the runtimes of our algorithm and integer program on $P_m \Box P_m$ graphs with one leak for values of $m = 3, 4, 5,$ and $6$. Our results are detailed in Table \ref{boxruntimes}. Each of these four graphs has the same structure, so the dramatic increase in runtime shows that our integer program is sensitive to the number of vertices in the graph.

\def\arraystretch{1.5}
\begin{center}
\begin{table} [h]
    \begin{tabular}{| c | c | c | c |}
    \hline
    $m$ & $|V|$ & $Z_{(1)}(G)$  & Time (seconds)
    \\ \hline
    3 & 9 & 3 & 0.7851
    \\ \hline
    4 & 16 & 4 & 7.9365
    \\ \hline
    5 & 25 & 5 & 101.3456
    \\ \hline
    6 & 36 &  6 & 1134.9772
    \\ \hline
    \end{tabular}
\caption{\label{tab: array run times} Run times finding $Z_{(1)}(P_m \Box P_m)$.}
\label{boxruntimes}
\end{table}
\end{center}

We also tested the runtimes of our algorithm and integer program on nine cubic graphs from the Wolfram database \cite{wolfram} with one leak, the results of which are found in Table \ref{cubicruntimes}. Cubic graphs are those in which all vertices have a degree of 3. Since these nine graphs have very similar structures, we can more easily compare their runtimes. Just as with the $P_m \Box P_m$ graphs, we can see that the number of vertices does affect the runtime. However, within a fixed number of vertices, smaller zero forcing numbers have smaller runtimes. Since our method begins with the smallest possible $\ell$-forcing set and iteratively adds more vertices to this set as it runs, this increase in time makes sense.

A comparable method was timed on cubic graphs with no leaks in \cite{compzf}. Their results for cubic graphs on 20 vertices yielded an average runtime of 0.40 seconds, and for cubic graphs on 30 vertices yielded an average runtime of 2.05 seconds. Although our runtimes are not nearly as fast, our algorithm took into account the potential for a leak to be on any vertex in the graph, increasing the complexity of our process.

\def\arraystretch{1.5}
\begin{center}
\begin{table} [h]
    \begin{tabular}{| c | c | c | c |}
    \hline
    Graph name & $|V|$ & $Z_{(1)}(G)$  & Time (seconds)
    \\ \hline
    Cubic\_20\_1 & 20 & 6 & 101.2786
    \\ \hline
    Cubic\_20\_2 & 20 & 6 & 75.9405
    \\ \hline
    Cubic\_20\_3 & 20 & 7 & 223.4605
    \\ \hline
    Cubic\_22\_1 & 22 & 7 & 392.8827
    \\ \hline
    Cubic\_22\_2 & 22 & 7 & 301.0180
    \\ \hline
    Cubic\_24\_1 & 24 & 8 & 1058.7689
    \\ \hline
    Cubic\_24\_2 & 24 & 6 & 481.3399
    \\ \hline
    Cubic\_24\_3 & 24 & 8 & 927.1612
    \\ \hline
    Cubic\_24\_4 & 24 & 8 & 968.2742
    \\ \hline
    \end{tabular}

\caption{\label{tab:cube run times} Run times finding $Z_{(1)}(G)$ for various cubic graphs.}
\label{cubicruntimes}
\end{table}
\end{center}

\section{Conclusion}

We introduced zero forcing and $\ell$-forcing and provided the $\ell$-forcing numbers of various families of graphs, such as paths, cycles, complete graphs, and wheels. Additionally, we explored $\ell$-forcing on hypercubes and provided a $1$-forcing set for any $P_n \Box P_m$ graph. We used the idea of forts in graphs as the foundation for our algorithm and integer program which find the $\ell$-forcing number for any finite, connected graph.

For future work, one goal would be to gain clarity on the patterns of interactions among the dimension of a cube graph, the number of leaks on the graph, and the $\ell$-forcing number of the graph. Another goal would be to find, if one exists, a minimum size pattern of vertices which color any $P_n \Box P_m$ graph. It appears that for $m \ge n$, $Z_{(1)}(P_n \Box P_m) = \min(m,2n)$. However, a universal pattern that applies for any $m$ and $n$ is elusive. Based on our work, if there is an example for which $Z_{(1)}(P_n \Box P_m) > \min(m,2n)$, $n$ must be at least $10$ which is beyond our current computational capabilities.

The second author is supported, in part, by NSF Grant DMS-1719894 and ONR grant N00014-18-W-X00709.

\bibliographystyle{siam}
\bibliography{copsandrobbers}


\end{document}